\documentclass[11pt,a4paper,twoside,%leqno
]{article}
\usepackage[french,english]{babel}%pour anglais prioritaire
\usepackage{amsfonts}
\usepackage{amsmath,amssymb,amscd}%% avec amscd et l'environnement CD, on peut faire des diagrammes.
\usepackage{mathrsfs}  %???
\usepackage{amsthm}%%
\usepackage{a4wide}   % activer pour une page agrandie
\usepackage[T1]{fontenc} % % accents dans le DVI
\usepackage{newlfont} %??

\usepackage[%dvipdfm,colorlinks,
backref=page,bookmarksopen=true]{hyperref} % pour faire des références dynamiques et indiquer où est cité chaque élément de la biblio

\usepackage[all]{xy}  % pour dessiner des beaux diagrammes

\def\choixcompteur{subsection}
\newtheorem{theo}[\choixcompteur]{Theorem}

\newtheorem{prop}[\choixcompteur]{Proposition}
\newtheorem{lemm}[\choixcompteur]{Lemma}
\newtheorem{coro}[\choixcompteur]{Corollary}

\theoremstyle{definition}
\newtheorem{defiprop}[\choixcompteur]{Definition-Proposition}
\newtheorem{defi}[\choixcompteur]{Definition}

\newtheorem{rema}[\choixcompteur]{Remark}
\newtheorem{remas}[\choixcompteur]{Remarks}

\newtheorem*{exem*}{Example}
\newtheorem*{exems*}{Examples}
\newtheorem*{exam*}{Example}
\newtheorem*{exams*}{Examples}
\newtheorem*{rema*}{Remark}
\newtheorem*{remas*}{Remarks}
\newtheorem*{NB}{N.B}

\theoremstyle{definition}
\newtheorem*{defi*}{Definition}
\newtheorem*{defiprop*}{Definition-Proposition}

\theoremstyle{plain}
\newtheorem*{prop*}{Proposition}
\newtheorem*{lemm*}{Lemma}
\newtheorem*{coro*}{Corollary}
\newtheorem*{theo*}{Theorem}

 \def\cdr@enoncedef{%
 \newenvironment{enonce*}[2][plain]%
 {\let\cdrenonce\relax \theoremstyle{##1}%
 \newtheorem*{cdrenonce}{##2}%
 \begin{cdrenonce}}%
 {\end{cdrenonce}}   }%

 \AtBeginDocument{%
    \cdr@enoncedef}

\def\cf{{\it cf.\/}\ }
\def\ie{{\it i.e.\/}\ }
\def\eg{{\it e.g.\/}\ }

\def\lc{{\it l.c.\/}\ }

\def\vect{\overrightarrow}
\def\rest{\,\rule[-1.5mm]{.1mm}{3mm}_}  %restriction  ,rajouter apr\`es :    {\,blabla}

\def\N{{\mathbb N}}    %lettres grasses sp\'eciales: \H, \L, \O, \P et \S d\'ej\`a pris
\def\Z{{\mathbb Z}}

\def\R{{\mathbb R}}
\def\C{{\mathbb C}}

\def\F{{\mathbb F}}
\def\A{{\mathbb A}}
\def\M{{\mathbb M}}
\def\T{{\mathbb T}}

\newcommand{\g}[1]{\mathfrak{#1}} %lettres gothiques

\def\qa{\alpha}     % lettres grecques
\def\qb{\beta}
\def\qd{\delta}
\def\qe{\varepsilon}
\def\qf{\varphi}
\def\qg{\gamma}

 \def\qk{\kappa}
 \def\ql{\lambda}
\def\qm{\mu}
\def\qn{\nu}
\def\qo{\omega}
\def\qp{\pi}
\def\qr{\rho}

\def\qt{\tau}
 \def\qth{\theta}
 
\def\qx{\xi}
 \def\qz{\zeta}

\def\QD{\Delta}
\def\QF{\Phi}
\def\QG{\Gamma}
\def\QL{\Lambda}
\def\QO{\Omega}

\def\shh{{\mathcal H}}

\def\shk{{\mathcal K}}
\def\shl{{\mathcal L}}
\def\shm{{\mathcal M}}

\def\sho{{\mathcal O}}

\def\shq{{\mathcal Q}}

\def\shs{{\mathcal S}}
\def\sht{{\mathcal T}}

\def\SHF{{\mathscr F}}

\def\SHI{{\mathscr I}}

\begin{document}

%\title[titre abrege]{Spherical Hecke algebras for Kac-Moody groups \goodbreak over local fields}  %% pour amsart?
\title{Spherical Hecke algebras for Kac-Moody groups \goodbreak over local fields}
\author{St\'ephane Gaussent and Guy Rousseau}

\date{May 22, 2012}%si absent, c'est la date du jour qui sera mise (ou pas de date dans "amsart"?)

\maketitle

%\author[S. Gaussent]{St\'ephane Gaussent*}
%\address{adresse}
%\email{gaussent@}
%\thanks{*ANRS}
%\urladdress{}
%
%\author[G. Rousseau]{Guy Rousseau**}
%\address{adresse}
%\email{rousseau@}
%\thanks{**ANRS}
%\urladdress{}

%\dedicatory{to him}   %% pour amsart

%\subjclass{17B67}   %% pour amsart
%\keywords{Hecke algebra, Kac--Moody group, local field}  %% pour amsart
%\altkeywords{alg\`ebre de Hecke , groupe de Kac--Moody, corps local}  %% pour smfart

%\par Version de travail \hskip 5em (\the\day \quad \the\month \quad \the\year) %% activer pour brouillons

%%%%%%%%%%%%%%%%%%%%%%%%%%%%%%%%%%%%%%%%%%%%%%%%%%%
%%%%%%%%%%%%%%%%%%%%%%%%%%%%%%%%%%%%%%%%%%%%%%%%%%%
%                                                TEXTE
%%%%%%%%%%%%%%%%%%%%%%%%%%%%%%%%%%%%%%%%%%%%%%%%%%%
%%%%%%%%%%%%%%%%%%%%%%%%%%%%%%%%%%%%%%%%%%%%%%%%%%%

%\bigskip

\begin{abstract}
We define the spherical Hecke algebra $\shh$ for an almost split Kac-Moody group $G$ over a local non-archimedean field. We use the hovel $\SHI$ associated to this situation, which is the analogue of the Bruhat-Tits building for a reductive group.
 The stabilizer $K$ of a special point on the standard apartment plays the role of a maximal open compact subgroup. We can define $\shh$ as the algebra of $K-$bi-invariant functions on $G$ with almost finite support. 
 As two points in the hovel are not always in a same apartment, this support has to be in some large subsemigroup $G^+$ of $G$. 
 We prove that the structure constants of $\shh$ are polynomials in the cardinality of the residue field, with integer coefficients depending on the geometry of the standard apartment.
 We also prove the Satake isomorphism between $\shh$ and the algebra of Weyl invariant elements in some completion of a Laurent polynomial algebra. In particular, $\shh$ is always commutative.
 Actually, our results apply to abstract ``locally finite'' hovels, so that we can define the spherical algebra with unequal parameters. 
\end{abstract}

%\begin{altabstract}    %% pour smfart
%\end{altabstract}

\setcounter{tocdepth}{1}    %%%%  = only sections 
\tableofcontents

\section*{Introduction}
\label{seIntro}

\par Let $G$ be a connected reductive group over a local non-archimedean field $\shk$ and let $K$ be an open compact subgroup. The space $\shh$ of complex functions on $G$, bi-invariant by $K$ and with compact support is an algebra for the natural convolution product.
 Ichiro Satake \cite{Sa63} studied this algebra $\shh$ to define the spherical functions and proved, in particular, that $\shh$ is commutative for good choices of $K$.
 We know now that one of the good choices for $K$ is the fixator of some special vertex for the action of $G$ on its Bruhat-Tits building $\SHI$, whose structure is explained in \cite{BrT72}.
 Moreover $\shh$, now called the spherical Hecke algebra, may be entirely defined using $\SHI$, see \eg \cite{P06}.
 
 \par Kac-Moody groups are interesting generalizations of reductive groups and it is natural to try to generalize the spherical Hecke algebra to the case of a Kac-Moody group.
  But there is now no good topology on $G$ and no good compact subgroup, so the ``convolution product'' has to be defined only with algebraic means.
  Alexander Braverman and David Kazhdan \cite{BrK10} succeeded in defining such a spherical Hecke algebra, when $G$ is split and untwisted affine.
  For a well chosen subgroup $K$, they define $\shh$ as an algebra of $K-$bi-invariant complex functions with ``almost finite'' support. There are two new features: the support has to be in a subsemigroup $G^+$ of $G$ and it is an infinite union of double classes.
  Hence, $\shh$ is naturally a module over the ring of complex formal power series.
  
  \par Our idea is to define this spherical Hecke algebra using the hovel associated to the almost split Kac-Moody group $G$ that we built in \cite{GR08}, \cite{R12} and \cite{R13}.
  This hovel $\SHI$ is a set with an action of $G$ and a covering by subsets called apartments. They are in one-to-one correspondence with the maximal split subtori, hence permuted transitively by $G$.
  Each apartment $A$ is a finite dimensional real affine space and its stabilizer $N$ in $G$ acts on it via a generalized affine Weyl group $W=W^v\ltimes Y$ (where $Y\subset\vect A$ is a discrete subgroup of translations) which stabilizes a set $\shm$ of affine hyperplanes called walls.
  So, $\SHI$ looks much like the Bruhat-Tits building of a reductive group, but $\shm$ is not a locally finite system of hyperplanes (as the root system $\QF$ is infinite) and two points in $\SHI$ are not always in a same apartment (this is why $\SHI$ is called a hovel).
  There is on $\SHI$ a $G-$invariant preorder $\leq$ which induces on each apartment $A$ the preorder given by the Tits cone $\sht\subset\vect A$.
  
  \par Now, we consider the fixator $K$ in $G$ of a special point $0$ in a chosen standard apartment $\A$. 
  The spherical Hecke algebra $\shh_R$ is a space of $K-$bi-invariant functions on $G$ with values in a ring $R$. In other words, it is the space $\shh_R^\SHI$ of $G-$invariant functions on $\SHI_0\times\SHI_0$ where $\SHI_0=G/K$ is the orbit of $0$ in $\SHI$.
  The convolution product is easy to guess from this point of view: 
  $(\qf*\psi)(x,y)=\sum_{z\in\SHI_0}\,\qf(x,z)\psi(z,y)$ (if this sum means something).
  As two points $x,y$ in $\SHI$ are not always in a same apartment (\ie the Cartan decomposition fails: $G\neq KNK$), we have to consider pairs $(x,y)\in\SHI_0\times\SHI_0$, with $x\leq y$ (this implies that $x,y$ are in a same apartment).
  For $\shh_R$, this means that the support of $\qf\in\shh_R$ has to be in $K\backslash G^+/K$ where $G^+=\{g\in G\mid 0\leq g.0\}$ is a semigroup.
  In addition, $K\backslash G^+/K$ is in one-to-one correspondence with the subsemigroup $Y^{++}=Y\cap C^v_f$ of $Y$ (where $C^v_f$ is the fundamental Weyl chamber).
  Now, to get a well defined convolution product, we have to ask (as in \cite{BrK10}) the support of a $\qf\in\shh_R$ to be almost finite: $supp(\qf)\subset\bigcup_{i=1}^n\,(\ql_i-Q_+^\vee)\cap Y^{++}$, where $\ql_i\in Y^{++}$ and $Q_+^\vee$ is the subsemigroup of $Y$ generated by the fundamental coroots. Note that $ (\ql-Q_+^\vee)\cap Y^{++}$ is infinite except when $G$ is reductive.
  
  \par With this definition we are able to prove that $\shh_R$ is really an algebra, which generalizes the known spherical Hecke algebras in the finite or affine split case (\S \ref{s2}).
  In the split case, we describe the hovel $\SHI$ and give a direct proof that  $\shh_R$ is commutative (\S \ref{s3}).
  
\par  The structure constants of $\shh_R$ are the non-negative integers $m_{\ql,\qm}(\qn)$ (for $\ql,\qm,\qn\in Y^{++}$) such that $c_\ql*c_\qm=\sum_{\qn\in Y^{++}}\,m_{\ql,\qm}(\qn)c_\qn$, where $c_\ql$ is the characteristic function of $K\ql K$.
  Each chamber (= alcove) in $\SHI$ has only a finite number of adjacent chambers along a given panel. These numbers are called parameters of $\SHI$ and they form a finite set $\shq$. In the split case, there is only one parameter $q$: the number of elements of the residue field $\qk$ of $\shk$.
  In \S \ref{s4} we show that the structure constants are polynomials in these parameters with integral coefficients depending only on the geometry of an apartment. 
  
  \par  In \S \ref{s5} we build an action of  $\shh_R$ on the module of functions from $\A\cap\SHI_0$ to $R$. This gives an injective homomorphism from $\shh_R$ into a suitable completion $R[[Y]]$ of the group algebra $R[Y]$; hence $\shh_R$ is abelian (\ref{5.3}).
  After modification by a character this homomorphism gives the Satake isomorphism from $\shh_R$ onto the subalgebra $R[[Y]]^{W^v}$ of $W^v-$invariant elements in $R[[Y]]$.
  The proof involves a parabolic retraction of $\SHI$ onto an extended tree inside it.
  
  \par Actually, this article is written in a more general framework (explained in \S \ref{s1}): we ask $\SHI$ to be an abstract ordered hovel (as defined in \cite{R11}) and $G$ a strongly transitive group of (positive, type-preserving) automorphisms.
  
  \par The general definition and study of Hecke algebras for split Kac-Moody groups over local fields was also undertaken by Alexander Braverman, David Kazhdan and Manish Patnaik (as we knew from \cite{P10}).
  A preliminary draft appeared recently \cite{BrKP12}. Their arguments are algebraic without use of a geometric object as a hovel, and the proofs seem complete (temporarily?) only for the untwisted affine case.
  In addition to the construction of the spherical Hecke algebra and the Satake isomorphism (as here),   they give a formula for spherical functions and they build the Iwahori-Hecke algebra. We hope to generalize, in a near future, these results to our general framework.
  
  \par One should notice that these authors use, instead of our group $K$, a smaller $K_1$, a priori slightly different, see Remark in Section \ref{3.4}. 

%%%%%%%%%%%%%%%%%%%%%%%%%%%%%%%%%%%%%%%%%%%%
\section{General framework}\label{s1}

\subsection{Vectorial data}\label{1.1}  We consider a quadruple $(V,W^v,(\qa_i)_{i\in I}, (\qa^\vee_i)_{i\in I})$ where $V$ is a finite dimensional real vector space, $W^v$ a subgroup of $GL(V)$ (the vectorial Weyl group), $I$ a finite set, $(\qa^\vee_i)_{i\in I}$ a family in $V$ and $(\qa_i)_{i\in I}$ a free family in the dual $V^*$.
 We ask these data to verify the conditions of \cite[1.1]{R11}.
  In particular, the formula $r_i(v)=v-\qa_i(v)\qa_i^\vee$ defines a linear involution in $V$ which is an element in $W^v$ and $(W^v,\{r_i\mid i\in I\})$ is a Coxeter system.

  \par To be more concrete we consider the Kac-Moody case of [\lc; 1.2]: the matrix $\M=(\qa_j(\qa_i^\vee))_{i,j\in I}$ is a generalized Cartan matrix.
  Then $W^v$ is the Weyl group of the corresponding Kac-Moody Lie algebra $\g g_\M$ and the associated real root system is
$$
\QF=\{w(\qa_i)\mid w\in W^v,i\in I\}\subset Q=\bigoplus_{i\in I}\,\Z.\qa_i.
$$ We set $\QF^\pm{}=\QF\cap Q^\pm{}$ where $Q^\pm{}=\pm{}(\bigoplus_{i\in I}\,(\Z_{\geq{}0}).\qa_i)$ and $Q^\vee=(\bigoplus_{i\in I}\,\Z.\qa_i^\vee)$, $Q^\vee_\pm{}=\pm{}(\bigoplus_{i\in I}\,(\Z_{\geq{}0}).\qa_i^\vee)$.
   We have  $\QF=\QF^+\cup\QF^-$ and, for $\qa=w(\qa_i)\in\QF$, $r_\qa=w.r_i.w^{-1}$ and $\qa^\vee=w(\qa_i^\vee)$ depend only on $\qa$, and $r_\qa(v)=v-\qa(v)\qa^\vee$.

\par The set $\QF$ is an (abstract reduced) real root system in the sense of \cite{MP89}, \cite{MP95} or \cite{Ba96}.
We shall sometimes also use the set $\QD=\QF\cup\QD^+_{im}\cup\QD^-_{im}$ of all roots (with $-\QD^-_{im}=\QD^+_{im}\subset Q^+$,  $W^v-$stable) defined in \cite{K90}.
 It is an (abstract reduced) root system in the sense of \cite{Ba96}.

  \par The {\it fundamental positive chamber} is $C^v_f=\{v\in V\mid\qa_i(v)>0,\forall i\in I\}$.
   Its closure $\overline{C^v_f}$ is the disjoint union of the vectorial faces $F^v(J)=\{v\in V\mid\qa_i(v)=0,\forall i\in J,\qa_i(v)>0,\forall i\in I\setminus J\}$ for $J\subset I$.
    The positive (resp. negative) vectorial faces are the sets $w.F^v(J)$ (resp. $-w.F^v(J)$) for $w\in W^v$ and $J\subset I$.
    The set $J$ or the face $w.F^v(J)$ is called {\it spherical} if the group $W^v(J)$ generated by $\{r_i\mid i\in J\}$ is finite.

    \par The {\it Tits cone}  $\sht$ is the (disjoint) union of the positive vectorial faces. It is a $W^v-$stable convex cone in $V$.

\subsection{The model apartment}\label{1.2} As in \cite[1.4]{R11} the model apartment $\A$ is $V$ considered as an affine space and endowed with a family $\shm$ of walls. %and a family $\shm^i$ of imaginary walls.
 These walls  are affine hyperplanes directed by Ker$(\qa)$ for $\qa\in\QF$.
%  (\footnote{Peut-\^etre pourra-t-on passer sous silence les murs imaginaires; \`a voir apr\`es r\'edaction compl\`ete.})

 \par We ask this apartment to be {\bf semi-discrete} and the origin $0$ to be {\bf special}.
  This means that these walls are the hyperplanes defined as follows: 
$$M(\qa,k)=\{v\in V\mid\qa(v)+k=0\}\qquad\text{for }\qa\in\QF\text{ and } k\in\QL_\qa$$ (with $\QL_\qa=k_\qa.\Z$ a non trivial discrete subgroup of $\R$). % and $\shm^i$ is defined similarly with $\qa\in\QD_{im}$ and $\QL_\qa$ a subset of $\R$ containing $0$. 
Using the following lemma (\ie replacing $\QF$ by $\widetilde\QF$) we shall assume that $\QL_\qa=\Z, \forall\qa\in\QF$.

  \par For $\qa=w(\qa_i)\in\QF$, $k\in\QL_\qa(=\Z)$ and $M=M(\qa,k)$, the reflection $r_{\qa,k}=r_M$ with respect to $M$ is the affine involution of $\A$ with fixed point set the wall $M$ and associated linear involution $r_\qa$.
   The affine Weyl group $W^a$ is the group generated by the reflections $r_M$ for $M\in \shm$; we assume that $W^a$ stabilizes $\shm$.
   
   \par For $\qa\in\QF$ and $k\in\R$, $D(\qa,k)=\{v\in V\mid\qa(v)+k\geq0\}$ is an half-space, it is called an {\it half-apartment} if $k\in\QL_\qa$ ($=\Z$).

%\par With these data we can define in $\A$ half-apartments, sectors, sector-faces, faces, enclosures... as in \cite{GR08} or \cite{R11}; see also \cite{R12}, \cite{R13}.
 
The Tits cone $\mathcal T$ %, its closure $\overline{\mathcal T}$ 
and its interior $\mathcal T^o$ are convex and $W^v-$stable cones, therefore, we can define two $W^v-$invariant preorder relations  on $\mathbb A$: 
$$
x\leq y\;\Leftrightarrow\; y-x\in\mathcal T
%;\quad x\overline{\leq} y\;\Leftrightarrow\; y-x\in\overline{\mathcal T}
; \quad x\stackrel{o}{\leq} y\;\Leftrightarrow\; y-x\in\mathcal T^o.
$$ 
 If $W^v$ has no  fixed point in $V\setminus\{0\}$ and no finite factor, then they are orders; but they are not in general.

   \begin{lemm}\label{1.2a} For all $\qa\in\QF$ we choose $k_\qa>0$ and define $\widetilde\qa=\qa/k_\qa$, $\widetilde\qa^\vee=k_\qa.\qa^\vee$.
   Then $\widetilde\QF=\{\widetilde\qa\mid\qa\in\QF\}$ is the (abstract reduced) real root system (in the sense of \cite{MP89}, \cite{MP95} or \cite{Ba96}) associated to $(V,W^v,(k_{\qa_i}^{-1}.\qa_i)_{i\in I}, (k_{\qa_i}.\qa^\vee_i)_{i\in I})$ hence to the generalized Cartan matrix $\widetilde\M=(k_{\qa_j}^{-1}.\qa_j(k_{\qa_i}.\qa_i^\vee))_{i,j\in I}$.
   Moreover with $\widetilde\QF$, the walls are described using the subgroups $\widetilde\QL_\qa=\Z$.
   \end{lemm}
 %  \begin{rema*}
%   \end{rema*}
   \begin{proof} For $\qa,\qb\in\QF$, the group $W^a$ contains the translation $\qt$ by $k_\qa.\qa^\vee$ and $\qt(M(\qb,0))=M(\qb,-\qb(k_\qa.\qa^\vee))$.
   So $k_\qa.\qb(\qa^\vee)\in\QL_\qb$ \ie $\widetilde\qb(\widetilde\qa^\vee)=k_{\qb}^{-1}.k_\qa.\qb(\qa^\vee)\in\Z$.
   Hence $\widetilde\M=(k_{\qa_j}^{-1}.\qa_j(k_{\qa_i}.\qa_i^\vee))_{i,j\in I}$ is a generalized Cartan matrix and the lemma is clear, as $k_{w\qa}=k_\qa$.
   \end{proof}

\subsection{Faces, sectors, chimneys...}
\label{suse:Faces}

 The faces in $\mathbb A$ are associated to the above systems of walls 
and halfapartments (\ie $D(\alpha,k) = \{ v\in \mathbb A\,\mid \, \alpha(v)+k \geq 0 \}$). As in \cite{BrT72}, they 
are no longer subsets of $\mathbb A$, but filters of subsets of $\mathbb A$. For the definition of that notion and its properties, we refer to \cite{BrT72} or \cite{GR08}.

If $F$ is a subset of $\mathbb A$ containing an element $x$ in its closure, 
the germ of $F$ in $x$ is the filter ${\mathrm germ}_x(F)$ consisting of all subsets of $\mathbb A$ which are intersections of $F$ and neighbourhoods of $x$. In particular, if $x\neq y\in E$, we denote the germ in $x$ of the segment $[x,y]$ (resp. of the interval $]x,y]$) by $[x,y)$ (resp. $]x,y)$).

The {\it enclosure} $cl_{\mathbb A}(F)$ of a filter $F$ of subsets of $\mathbb A$  is the filter made of  the subsets of $\mathbb A$ containing an element of $F$ of the shape $\cap_{\alpha\in\Delta}D(\alpha,k_\alpha)$, where $k_\alpha\in\mathbb Z\cup\{\infty\}$ (here, $D(\alpha,\infty) = \mathbb A$).

\medskip
A {\it face} $F$ in the apartment $\mathbb A$ is associated
 to a point $x\in \mathbb A$ and a  vectorial face $F^v$ in $V$; 
it is called spherical according to the nature of $F^v$. More 
precisely,  a subset $S$ of $\mathbb A$ is an element of the face $F(x,F^v)$ if and only if
it contains  an intersection of half-spaces $D(\alpha,k)$ or open halfspaces $D^\circ(\alpha,k)$ (for
$\alpha\in\Delta$ and $k\in\mathbb Z\sqcup\{\infty\}$) which contains
 $\Omega\cap(x+F^v)$, where $\Omega$ is an open neighborhood of $x$ in $\mathbb A$. The 
enclosure of a face $F=F(x,F^v)$ is its closure: the closed-face $\overline F$. It is the  enclosure of the local-face in $x$, ${\mathrm germ}_x(x+F^v)$. 

There is an order on the faces: the assertions ``$F$ is a face of $F'$ '', 
``$F'$ covers $F$ '' and ``$F\leq F'$ '' are by definition equivalent to
$F\subset\overline{F'}$.  The dimension of a face $F$ is the smallest dimension of an affine space generated by some $S\in F$. The (unique) such affine space $E$ of minimal dimension is the support of $F$. Any $S\in F$ contains a non empty open subset of $E$. A face $F$ is spherical if the direction of its support meets the open Tits cone, then its fixator $W_F$ in $W$ is finite.

 Any point $x\in \mathbb A$ is contained in a unique face $F(x,V_0)$ which is minimal (but seldom spherical); $x$ is a vertex if, and only if, $F(x,V_0)=\{x\}$. 

 A {\it chamber} (or  alcove) is a maximal face, or, equivalently, a face  such
that all its elements contain a nonempty open subset of $\mathbb A$.

A {\it panel} is a spherical face maximal among faces which are not chambers, or, equivalently, a spherical face of dimension $n-1$. Its support is a wall.
So, the set  of spherical faces of $\mathbb A$ and the Tits cone completely determine the set $\mathcal M$ of walls.

\medskip
 A {\it sector} in $\mathbb A$ is a $V-$translate $\mathfrak s=x+C^v$ of a vectorial chamber 
$C^v=\pm w.C^v_f$ ($w \in W^v$), $x$ is its {\it base point} and $C^v$ its  {\it direction}.  Two sectors have the same direction if, and only if, they are conjugate
by $V-$translation,
 and if, and only if, their intersection contains another sector. 

 The {\it sector-germ} of a sector $\mathfrak s=x+C^v$ in $\mathbb A$ is the filter $\mathfrak S$ of 
subsets of~$\mathbb A$ consisting of the sets containing a $V-$translate of $\mathfrak s$, it is well 
determined by the direction $C^v$. So the set of 
translation classes of sectors in $\mathbb A$, the set of vectorial chambers in $V$ and 
 the set of sector-germs in $\mathbb A$ are in canonical bijection.
  We write $\g S_{-\infty}$ the sector-germ associated to the negative fundamental vectorial chamber $-C^v_f$.

 A {\it sector-face} in $\mathbb A$ is a $V-$translate $\mathfrak f=x+F^v$ of a vectorial face
$F^v=\pm wF^v(J)$. The sector-face-germ of $\mathfrak f$ is the filter $\mathfrak F$ of 
subsets containing a translate $\mathfrak f'$ of $\mathfrak f$ by an element of $F^v$ ({\it i.e.} $\mathfrak
f'\subset \mathfrak f$). If $F^v$ is spherical, then $\mathfrak f$ and $\mathfrak F$ are also called
spherical. The sign of $\mathfrak f$ and $\mathfrak F$ is the sign of $F^v$.

\medskip
A {\it chimney} in $\mathbb A$ is associated to a face $F=F(x, F_0^v)$, its basis, and to a vectorial face $F^v$, its direction, it is the filter
$$
\mathfrak r(F,F^v) = cl_{\mathbb A}(F+F^v).
$$ A chimney $\mathfrak r = \mathfrak r(F,F^v)$ is {\it splayed} if $F^v$ is spherical, it is {\it solid} if its support (as a filter, i.e. the smallest affine subspace containing $\mathfrak r$) has a finite fixator in $W^v$. A splayed chimney is therefore solid. The enclosure of a sector-face $\mathfrak f=x+F^v$ is a chimney. 

 \par A halfline $\delta$ with origin in $x$ and containing $y\not=x$ (or the interval $]x,y]$, the segment $[x,y]$) is called {\it preordered} if $x\leq y$ or $y\leq x$ and {\it generic} if $x\stackrel{o}{\leq} y$ or $y\stackrel{o}{\leq} x$. With these new notions, a chimney can be defined as the enclosure of a preordered halfline and a preordered segment-germ sharing the same origin. The chimney is splayed if, and only if, the halfline is generic.

 \subsection{The hovel}\label{1.3}  
 
 In this section, we recall the definition of an ordered affine hovel given by Guy Rousseau in \cite{R11}. 

An apartment of type $\mathbb A$ is a set $A$ endowed with a set $Isom(\mathbb A,A)$ of bijections (called isomorphisms) such that if $f_0\in Isom(\mathbb A,A)$, then $f\in Isom(\mathbb A,A)$ if, and only if, there exists $w\in W^a$ satisfying $f = f_0\circ w$.
An isomorphism between two apartments $\phi :A\to A'$ is a bijection such that $f\in Isom(\mathbb A,A)$ if, and only if, $\phi\circ f\in Isom(\mathbb A,A')$.
 As the filters in $\A$ defined in \ref{suse:Faces} above (\eg faces, sectors, walls,..) are permuted by $W^a$, they are well defined in any apartment of type $\A$.

\begin{defi*}
\label{de:AffineHovel}
An ordered affine hovel of type $\mathbb A$ is a set $\SHI$ endowed with a covering $\mathcal A$ of subsets  called apartments such that: 
\begin{enumerate}
\item[{\bf (MA1)}] any $A\in \mathcal A$ admits a structure of an apartment of type $\mathbb A$;
\item[{\bf (MA2)}] if $F$ is a point, a germ of a preordered interval, a generic halfline or a solid chimney in an apartment $A$ and if $A'$ is another apartment containing $F$, then $A\cap A'$ contains the enclosure $cl_A(F)$ of $F$ and there exists an isomorphism from $A$ onto $A'$ fixing $cl_A(F)$;
\item[{\bf (MA3)}] if $\mathfrak R$ is a germ of a splayed chimney and if $F$ is a face or a germ of a solid chimney, then there exists an apartment that contains $\mathfrak R$ and $F$;
\item[{\bf (MA4)}] if two apartments $A,A'$ contain $\mathfrak R$ and $F$ as in {\bf (MA3)}, then their intersection contains $cl_A(\mathfrak R\cup F)$ and there exists an isomorphism from $A$ onto $A'$ fixing $cl_A(\mathfrak R\cup F)$;
\item[{\bf (MAO)}] if $x,y$ are two points contained in two apartments $A$ and $A'$, and if $x\leq_A y$ then the two segments $[x,y]_A$ and $[x,y]_{A'}$ are equal.
\end{enumerate}
\end{defi*}

  \par We ask here $\SHI$ to be thick of {\bf finite thickness}: the number of chambers (=alcoves) containing a given panel has to be finite $\geq{}3$.
     This number is the same for any panel in a given wall $M$ \cite[2.9]{R11}; we denote it by $1+q_M$.

\medskip
 We assume that $\SHI$ has a strongly transitive group of automorphisms $G$ (\ie all isomorphisms involved in the above axioms are induced by elements of $G$, \cf \cite[4.10]{R13}).
  We choose in $\SHI$ a fundamental apartment which we identify with $\A$.
   As $G$ is strongly transitive, the apartments of $\SHI$ are the sets $g.\A$ for $g\in G$. The stabilizer $N$ of $\A$ in $G$ induces a group $\qn(N)$ of affine automorphisms of $\A$ which permutes the walls, sectors, sector-faces... and contains the affine Weyl group $W^a$ \cite[4.13.1]{R13}.
   We denote the fixator of $0\in\A$ in $G$ by $K$.

     \par We ask  $\qn(N)$ to be {\bf positive} and {\bf type-preserving} for its action on the vectorial faces.
      This means that the associated linear map $\vect w$ of any $w\in\qn(N)$ is in $W^v$.
      As $\qn(N)$ contains $W^a$ and stabilizes $\shm$, we have $\qn(N)=W^v\ltimes Y$, where $W^v$ fixes the origin $0$ of $\A$ and $Y$ is a group of translations such that:
  \qquad    $Q^\vee\subset Y\subset P^\vee=\{v\in V\mid\qa(v)\in\Z,\forall\qa\in\QF\}$.
%  But we may consider the more general case where $G$ or $\qn(N)$ is {\bf type-permuting}.
%  Then $\qn(N)=(D\ltimes W^v)\ltimes Y$ with $W^v$, $Y$ as above and $D$ a finite subgroup fixing $0$, stabilizing $C^v_f$ and identified with a group of (diagram) permutations of $I$ by its action on the facets $F^v(J)$ of $C^v_f$.
%  There is moreover a normal subgroup $G^\circ$ of $G$ which is strongly transitive and type preserving with $D\simeq G/G^\circ\simeq K/(K\cap G^\circ)$.

  \par We ask $Y$ to be {\bf discrete} in $V$. This is clearly satisfied if $\QF$ generates $V^*$ \ie $(\qa_i)_{i\in I}$ is a basis of $V^*$.
 
 \medskip
  \par\noindent{\bf Examples.} The main examples of all the above situation are provided by the hovels of almost split Kac-Moody groups over fields complete for a discrete valuation and with a finite residue field, see \cite{R12}, \cite{Ch10}, \cite{Ch11} or \cite{R13}. Some details in the split case can be found in Section \ref{s3}.
% These Kac-Moody examples are type-preserving. We get a non trivial group $D$ of diagram automorphisms when considering a loop group $G$ associated to a non simply connected semi simple group over a local field; then $G^\circ$ is a Kac-Moody group.

\medskip
  \par\noindent{\bf Remarks.} a) In the following, we often refer to \cite{GR08} which deals with split Kac-Moody groups and residue fields containing $\C$. But the results cited are easily generalized to our present framework, using the above references.

  \par b) For an almost split Kac-Moody group over a local field $\shk$, the set of roots $\QF$ is $^\shk\QF_{red}=\{{^\shk\qa}\in{^\shk\QF}\mid \frac{1}{2}.{^\shk\qa}\not\in{^\shk\QF}\}$ where the relative root system $^\shk\QF$ describes well the commuting relations between the root subgroups.
  Unfortunately $\widetilde\QF$ gives a worst description of these relations.

  \subsection{Type $0$ vertices}\label{1.4} The elements of $Y$ considered as the subset $Y=N.0$ of $V=\A$ are called {\it vertices of type $0$} in $\A$; they are special vertices. We note $Y^+=Y\cap\sht$ and $Y^{++}=Y\cap \overline{C^v_f}$.
   The type $0$ vertices in $\SHI$ are the points on the orbit $\SHI_0$ of $0$ by $G$. This set $\SHI_0$ is often called the affine Grassmannian as it is equal to $G/K$.

   \par In general, $G$ is not equal to $KYK=KNK$ \cite[6.10]{GR08} \ie $\SHI_0\not=K.Y$.

   \par We know that $\SHI$ is endowed with a $G-$invariant preorder $\leq{}$ which induces the known one on $\A$ \cite[5.9]{R11}.
   We set $\SHI^+=\{x\in\SHI\mid0\leq{}x\}$ , $\SHI^+_0=\SHI_0\cap\SHI^+$ and $G^+=\{g\in G\mid0\leq{}g.0\}$; so $\SHI^+_0=G^+.0=G^+/K$.
   As $\leq{}$ is a $G-$invariant preorder, $G^+$ is a semigroup.

 \par  If $x\in\SHI^+_0$ there is an apartment $A$ containing $0$ and $x$ (by definition of $\leq{}$) and all apartments containing $0$ are conjugated to $\A$ by $K$ (axiom (MA2)); so $x\in K.Y^+$ as $\SHI^+_0\cap\A=Y^+$.
    But $\qn(N\cap K)=W^v$ and $Y^+=W^v.Y^{++}$ (with uniqueness of the element in $Y^{++}$); %(\footnote{probleme dans le cas type-permuting})
     so $\SHI^+_0=K.Y^{++}$, more precisely $\SHI^+_0=G^+/K$ is the disjoint union of the $KyK/K$ for $y\in Y^{++}$.

    \par Hence, we have proved that the map $Y^{++}\to K\backslash G^+/K$ is one-to-one and onto.

    \subsection{Vectorial distance and $Q^\vee-$order}\label{1.5}
    For $x\in\sht$, we note $x^{++}$ the unique element in $\overline{C^v_f}$ conjugated by $W^v$ to $x$.

    \par  Let $\SHI\times_\leq{}\SHI=\{(x,y)\in\SHI\times\SHI\mid x\leq{}y\}$ be the set of increasing pairs in $\SHI$.
    Such a pair $(x,y)$ is always in a same apartment $g.\A$; so $g^{-1}y-g^{-1}x\in\sht$ and we define the {\it vectorial distance} $d^v(x,y)\in  \overline{C^v_f}$ by $d^v(x,y)=(g^{-1}y-g^{-1}x)^{++}$.
    It does not depend on the choices we made.

    \par For $(x,y)\in \SHI_0\times_\leq{}\SHI_0=\{(x,y)\in\SHI_0\times\SHI_0\mid x\leq{}y\}$, the vectorial distance $d^v(x,y)$ takes values in $Y^{++}$.
     Actually, as $\SHI_0=G.0$, $K$ is the fixator of $0$ and $\SHI^+_0=K.Y^{++}$ (with uniqueness of the element in $Y^{++}$), the map $d^v$ induces a bijection between the set $\SHI_0\times_\leq{}\SHI_0/G$ of orbits of $G$ in $\SHI_0\times_\leq{}\SHI_0$ and $Y^{++}$.

     \par Any $g\in G^+$ is in $K.d^v(0,g0).K$.

     \par For $x,y\in\A$, we say that $x\leq{}_{Q^\vee}\,y$ (resp. $x\leq{}_{Q^\vee_\R}\,y$) when $y-x\in Q^\vee_+$ (resp. $y-x\in Q^\vee_{\R+}=\sum_{i\in I}\,\R_{\geq{}0}.\qa_i^\vee$).
     We get thus a preorder which is an order at least when $(\qa_i^\vee)_{i\in I}$ is free or $\R_+-$free (\ie $\sum a_i\qa_i^\vee=0,a_i\geq{}0\Rightarrow a_i=0,\forall i$).

\subsection{Paths}
\label{suse:Paths}

We consider piecewise linear continuous paths
$\pi:[0,1]\rightarrow \mathbb A$ such that each (existing) tangent vector $\pi'(t)$
 is in an orbit $W^v.\lambda$ of some $\lambda\in {\overline{C^v_f}}$ under the
vectorial Weyl group $W^v$. Such a path is called a {\it $\lambda-$path}; it is
increasing with respect to the preorder relation $\leq$ on $\mathbb A$.
% If $\pi(0)$, $\pi(1)$  and~$\lambda$  are in $Y$, we say that $\pi$ is ``in $Y$''.

 For any $t\neq 0$ (resp. $t \neq1$), we let
$\pi'_-(t)$ (resp. $\pi'_+(t)$) denote the derivative of $\pi$ at $t$ from the left
(resp. from the right). Further, we define $w_\pm(t)\in W^v$ to be the smallest
element  in its
$(W^v)_\lambda-$class such that $\pi'_\pm(t)=w_\pm(t).\lambda$ (where $(W^v)_\lambda$ is the fixator
in
$W^v$ of $\lambda$). Moreover, we denote by $\pi_-(t)=\pi(t)-[0,1)\pi_-'(t)=
[\pi(t),\pi(t-\varepsilon)\,)$ (resp. $\pi_+(t)=\pi(t)+[0,1)\pi_+'(t)=
[\pi(t),\pi(t+\varepsilon)\,)$ (for $\varepsilon>0$ small) the positive (resp. negative)
segment-germ of $\pi$ at $t$.

 The reverse path $\overline\pi$ defined by
${\overline\pi}=\pi(1-t)$ has symmetric properties, it is a $(-\lambda)-$path.

 For any choices of $\lambda\in {\overline{C^v_f}}$, $\pi_0\in \mathbb A$,
$r\in\mathbb N\setminus\{0\}$ and sequences
${\underline\tau}=(\tau_1,\tau_2,\dots,\tau_r)$ of elements in $W^v/(W^v)_\lambda$ and
${\underline a}=(a_0=0<a_1<a_2<\dots<a_r=1)$ of elements in $\mathbb R$, we define
a $\lambda-$path $\pi=\pi(\lambda,\pi_0,{\underline\tau},{\underline a})$
by the formula:
$$
\pi(t)=\pi_0+\sum_{i=1}^{j-1}\;(a_i-a_{i-1})\tau_i(\lambda)+(t-a_{j-1})\tau_j(\lambda)
\quad \hbox{ for } \quad a_{j-1}\leq t\leq a_j.
$$

\noindent Any $\lambda-$path may be defined in this way (and we may
 assume $\tau_j\neq\tau_{j+1}$).

\begin{defi*}
\label{de:Hecke}
\cite[3.27]{KM08}
A {\it Hecke path} of shape $\lambda$ with respect to $-C^v_f$ is a
$\lambda-$path such that, for all $t\in [0,1]\setminus\{0,1\}$,
$\pi_+'(t)\leq_{W^v_{\pi(t)}}\pi_-'(t)$, which means  that there exists a
${W^v_{\pi(t)}}-$chain from $\pi_-'(t)$ to $\pi_+'(t)$, {\it i.e.} finite sequences
$(\xi_0=\pi_-'(t),\xi_1,\dots,\xi_s=\pi_+'(t))$ of vectors in $V$ and
$(\beta_1,\dots,\beta_s)$  of  real roots such that, for all $i=1,\dots,s$:
\begin{itemize}
\item [i)] $r_{\beta_i}(\xi_{i-1})=\xi_i$,
\item[ii)] $\beta_i(\xi_{i-1})<0$,
\item[iii)] $r_{\beta_i}\in{W^v_{\pi(t)}}$ {\it i.e.} $\beta_i(\pi(t))\in\mathbb Z$:
$\pi(t)$ is in a wall of direction Ker$(\beta_i)$.
\item[iv)] each $\qb_i$ is positive  with respect to $-C^v_f$ \ie $\qb_i(C^v_f)>0$.
\end{itemize}
\end{defi*}

\begin{remas*} \label{1.9}
1) The path is folded at $\pi(t)$ by applying successive reflections along the walls $M(\beta_i,-\beta_i(\pi(t))\,)$. Moreover conditions ii) and iv)
tell us that the path is ``positively folded'' (cf. \cite{GL05}) \ie centrifugally folded with respect to the sector germ $\g S_{-\infty}=germ_\infty(-C^v_f)$.

\par 2) Let $\g c_-=germ_0(-C^v_f)$ be the negative fundamental chamber (= alcove).
A  {\it Hecke path} of shape $\lambda$ with respect to $\g c_-$ \cite{BCGR11} is a $\lambda-$path in the Tits cone $\sht$ satisfying the above conditions except that we replace iv) by :

\par iv') each $\qb_i$ is positive  with respect to $\g c_-$ \ie $\qb_i(\pi(t)-\g c_-)>0$.

\par Then ii) and iv') tell us that the path is centrifugally folded with respect to the center  $\g c_-$.

\end{remas*}

%%%%%%%%%%%%%%%%%%%%%%%%%%%%%%%%%%%%%%%%%%%%

\section{Convolution algebras}\label{s2}

\subsection{Wanted}\label{2.1} 
We consider the space 
$$\widehat\shh_R^\SHI=\widehat\shh_R(\SHI,G)=\{\qf^{\SHI\!} :\SHI_0\times_\leq{}\SHI_0\to R\mid\qf^{\SHI\!}(gx,gy)=\qf^{\SHI\!}(x,y),\forall g\in G\}
$$ of $G-$invariant functions on $\SHI_0\times_\leq{}\SHI_0$ with values in a ring $R$ (essentially $\C$ or $\Z$). We want to make $\widehat\shh_R^\SHI$ (or some large subspace) an algebra for the following convolution product:
$$
(\qf^{\SHI\!}*\psi^{\SHI\!})(x,y)=\sum_{x\leq{}z\leq{}y}\,\qf^{\SHI\!}(x,z)\psi^{\SHI\!}(z,y).
$$
%with perhaps a kind of integral instead of the sum.(\footnote{\`a voir, c'est plut\^ot dans le cas affine que je pensais \`a cette \'eventualit\'e.})
It is clear that this product is associative and $R-$bilinear if it exists.

 \par Via $d^v$, $\widehat\shh_R^\SHI$ is linearly isomorphic to the space $\widehat\shh_R=\{\qf^{G\!}:Y^{++}=K\backslash G^+/K\to R\}$, which can be interpreted as the space of $K-$bi-invariant functions on $G^+$.
  The correspondence $\qf^{\SHI\!}\leftrightarrow\qf^{G\!}$ between $\widehat\shh_R^\SHI$ and $\widehat\shh_R$ is given by: 
$$
\qf^{G\!}(g)=\qf^{\SHI\!}(0,g.0)\quad\text{and}\quad \qf^{\SHI\!}(x,y)= \qf^{G\!}(d^v(x,y)).
$$
 In this setting, the convolution product should be: $(\qf^{G\!}*\psi^{G\!})(g)=\sum_{h\in G^+/K}\, \qf^{G\!}(h)\psi^{G\!}(h^{-1}g)$, where we consider $\qf^{G\!}$ and $\psi^{G\!}$ as trivial on $G\setminus G^+$.
  %and the sum could perhaps be replaced by a kind of integral.
  In the following we shall often make no difference between $\qf^{\SHI\!}$ or  $\qf^{G\!}$ and forget the exponents $^{\SHI\!}$ and $^{G\!}$.
  
    We consider the subspace $\shh_R^f$ of functions with finite support in $Y^{++}=K\backslash G^+/K$; its natural basis is $(c_\ql)_{\ql\in Y^{++}}$ where $c_\ql$ sends $\ql$ to $1$ and $\qm\not=\ql$ to $0$. Clearly $c_0$ is a unit for $*$.
 In $\widehat\shh_R^\SHI$, $(c_\ql*c_\qm)^{\SHI\!}(x,y)$ is the number of triangles $[x,z,y]$ with $d^v(x,z)=\ql$ and $d^v(z,y)=\qm$.

  \par As suggested by \cite{BrK10} and lemma \ref{2.3}, we consider also the subspace $\shh_R$  of $\widehat\shh_R$ of functions $\qf$ with {\it almost finite}  support \ie $supp(\qf)\subset\cup_{i=1}^n\,(\ql_i-Q^\vee_+)\cap Y^{++}$ where $\ql_i\in Y^{++}$. %(\footnote{Je m'aper\c{c}ois qu'\`a partir d'ici je suppose souvent implicitement que $\QL_\qa=\Z,\forall\qa\in\QF$.})   (resp. $\qr(supp(\qf))$ bounded above, if $\qr\in V^*$ is such that $\qr(\qa_i^\vee)>0$, $\forall i\in I$). (\footnote{On pourrait peut-\^etre dire $Q^\vee_--$finite au lieu de almost finite et supprimer $\shh_R^{ba}$.})

  \subsection{Retractions onto $Y^+$}\label{2.2} % Let $\g c_-=germ_0(-C^v_f)$ be the negative fundamental chamber (=alcove).
   For all $x\in \SHI^+$ there is an apartment containing $x$ and $\g c_-$ \cite[5.1]{R11} and this apartment is conjugated to $\A$ by an element of $K$ fixing $\g c_-$ (axiom (MA2) ).
    So, by the usual arguments and [\lc, 5.5] we can define a retraction $\qr_{\g c_-}$ of $\SHI^+$ into $\A$ with center $\g c_-$; its image is $\qr_{\g c_-}(\SHI^+)=\sht=\SHI^+\cap\A$ and $\qr_{\g c_-}(\SHI^+_0)=Y^+$.

    \par There is also a retraction $\qr_{-\infty}$ of $\SHI$ onto $\A$ with center the sector-germ $\g S_{-\infty}$  \cite[4.4]{GR08}.

    \par For $\qr=\qr_{\g c_-}$ or $\qr_{-\infty}$ the image of a segment $[x,y]$ with $(x,y)\in\SHI\times_\leq{}\SHI$ and $d^v(x,y)=\ql\in\overline{C^v_f}$ is a $\ql-$path % \ie the concatenation of segments $[x_i,x_{i+1}]$ with $x_{i+1}-x_{i}=w_ia_i\ql$, $w_i\in W^v$, $a_i\in[0,1]$ and $\sum a_i=1$
    \cite[4.4]{GR08}. In particular, $\qr(x)\leq{}\qr(y)$.

   \subsection{Convolution product}\label{2.2a}

   \par The convolution product in $\widehat\shh_R$ should be defined (for $y\in Y^{++}$) by 
$$
(\qf*\psi)(y)=\sum\,\qf(z)\psi(d^v(z,y))
$$ where the sum runs over the $z\in\SHI^+_0$ such that $0\leq{}z\leq{}y$ and $\qf(z)=\qf^\SHI(0,z)=\qf^G(d^v(0,z))$.

\medskip
   \par 1) Using  $\qr_{\g c_-}$ we have, for $\ql,\qm,y\in Y^{++}$,   $(c_\ql*c_\qm)(y)=\sum_{w\in W^v/(W^v)_\lambda}\,N_{\g c_-}(\qm,w.\ql,y)$ where  $N_{\g c_-}(\qm,w.\ql,y)$ is the number of $z\in\SHI^+_0$ with $d^v(z,y)=\qm$ and $\qr_{\g c_-}(z)=w.\ql\in Y^+$.
   Note that, if $N_{\g c_-}(\qm,w\ql,y)>0$, there exists a $\qm-$path from $w\ql$ to $y$, hence $y\in w\ql+Y^+$.

   \par So $c_\ql*c_\qm$ is the formal sum $c_\ql*c_\qm=\sum_{\qn\in Y^{++}}\,m_{\ql,\qm}(\qn)c_\qn$ where the structure constant $m_{\ql,\qm}(\qn)=\sum_{w\in W^v/(W^v)_\lambda}\,N_{\g c_-}(\qm,w.\ql,\qn)\in\Z_{\geq{}0}\cup\{+\infty\}$ is also equal to the number of triangles $[x,z,y]$ with $d^v(x,z)=\ql$ and $d^v(z,y)=\qm$, for any fixed pair $(x,y)\in \SHI_0\times_\leq{}\SHI_0$ with $d^v(x,y)=\qn$ (\eg $(x,y)=(0,\qn)$).

\medskip
\par 2) Using $\qr_{-\infty}$ we have $m_{\ql,\qm}(\qn)=\sum_{z'}\,N_{-\infty}(\qm,z',\qn)$ where the sum runs over the $z'$ in $Y^+(\ql)=\qr_{-\infty}(\{z\in\SHI^+_0\mid d^v(0,z)=\ql\})$ and $N_{-\infty}(\qm,z',\qn)\in\Z_{\geq{}0}\cup\{+\infty\}$ is the number of $z\in\SHI^+_0$ with $d^v(0,z)=\ql$, $d^v(z,y)=\qm$ (for any $y\in\SHI^+_0$ with $d^v(0,y)=\qn$ \eg $y=\qn$) and $\qr_{-\infty}(z)=z'$.
 But $\qr_{-\infty}([0,z])$ is a $\ql-$path hence increasing with respect to $\leq{}$, so $Y^+(\ql)\subset Y^+$.
  Moreover, $\qr_{-\infty}([z,\qn])$ is a $\qm-$path, so $z'$ has to be in $\qn-Y^+$.
   Hence, $z'$ has to run over the set $Y^+(\ql)\cap(\qn-Y^+)\subset Y^+\cap(\qn-Y^+)$.

   \par Actually, the image by $\qr_{-\infty}$ of a segment $[x,y]$ with $(x,y)\in\SHI\times_\leq{}\SHI$ and $d^v(x,y)=\ql\in Y^{++}$ is a Hecke path of shape $\ql$ with respect to $-C^v_f$ \cite[th. 6.2]{GR08}. Hence the following results:

%   \par Actually $N(\qm,w.\ql,\qn)$ is finite (\`a v\'erifier) and will be calculated precisely in {\S{}}  \ref{s4};\cf \cite[5.9 and 6.3]{GR08} in a similar situation. But it is not clear that $N(\qm,w.\ql,\qn)=0$ for almost all $w\in W^ v$ \ie that $m_{\ql,\qm}(\qn)$ is finite. Another problem is that $m_{\ql,\qm}(\qn)$ is perhaps not trivial for almost all $\qn$, so $c_\ql*c_\qm$  would be really an infinite formal sum and $\shh_R$ not a subalgebra.

%   \par To deal with these problems we shall consider separately the three kinds of root systems $\QF$ (supposed indecomposable): finite, affine and indefinite.

 \begin{lemm}\label{2.3} a) For $\ql\in Y^{++}$ and $w\in W^v$, $w\ql\in\ql-Q^\vee_+$, \ie $w\ql\leq{}_{Q^\vee}\,\ql$.

 \par b) Let $\qp$ be a Hecke path of shape $\ql\in Y^{++}$ with respect to $-C^v_f$, from $y_0\in Y$ to $y_1\in Y$.
  Then $\ql=\qp'(0)^{++}=\qp'(1)^{++}$, $\qp'(0)\leq{}_{Q^\vee}\,\ql$, %$\qp'(1)\leq{}_{Q^\vee}\,\ql$
  $\qp'(0)\leq{}_{Q^\vee_\R}\,(y_1-y_0)\leq{}_{Q^\vee_\R}\,\qp'(1)\leq{}_{Q^\vee}\,\ql$ and $y_1-y_0 \leq{}_{Q^\vee}\,\ql$.
% In particular $Y^+(\ql)\subset\ql-Q^\vee_+$.

\par c) If moreover $(\qa_i^\vee)_{i\in I}$ is free, we may replace above $\leq{}_{Q^\vee_\R}$ by $\leq{}_{Q^\vee}$.

 \par d) For $\ql,\qm,\qn\in Y^{++}$, if $m_{\ql,\qm}(\qn)>0$, then $\qn\in\ql+\qm-Q^\vee_+$ \ie $\qn\leq{}_{Q^\vee}\,\ql+\qm$.
 \end{lemm}
\begin{NB} By d) above, if $x\leq{}z\leq{}y$ in $\SHI_0$, then $d^v(x,y)\leq{}_{Q^\vee}d^v(x,z)+d^v(z,y)$.
\end{NB}
\begin{proof} a) By definition, for $\ql\in Y$, $w\ql\in\ql+Q^\vee$, hence a) follows from \cite[3.12d]{K90} used in a realization where $(\qa_i^\vee)_{i\in I}$ is free.

\par b) By definition of Hecke paths in \ref{de:Hecke}, $\ql=\qp'(0)^{++}=\qp'(1)^{++}$.
Moreover, $\forall t\in[0,1]$, $\ql=\qp_-'(t)^{++}=\qp_+'(t)^{++}$
%But prop. 6.1 of \lc describes
and we know how to get $\qp_+'(t)$ from $\qp_-'(t)$ by successive reflections; this proves that $\qp'_+(t)\in\qp'_-(t)+Q^\vee_{\R+}$.
 By integrating the locally constant function $\qp'(t)$, we get $\qp'(0)\leq{}_{Q^\vee_\R}\,(y_1-y_0)\leq{}_{Q^\vee_\R}\,\qp'(1)\leq{}_{Q^\vee_\R}\,\ql$.

\par It is proved (but not stated) in \cite[5.3.3]{GR08} that any Hecke path of shape $\ql$ starting in $y_0\in Y$  can be transformed in the path $\qp_\ql(t)=y_0+\ql t$ by applying successively the operators $e_{\qa_i}$ or $\widetilde e_{\qa_i}$ for $i\in I$; moreover $e_{\qa_i}(\qp)(1)=\qp(1)+\qa_i^\vee$ and $\widetilde e_{\qa_i}(\qp)(1)=\qp(1)$, hence $y_1-y_0 \leq{}_{Q^\vee}\,\ql$.

\par c) By b) $y_1-y_0-\qp'(0)\in Q^\vee_{\R+}\cap Q^\vee=Q^\vee_{+}$, so $\qp'(0)\leq{}_{Q^\vee}\,(y_1-y_0)$. Idem for $y_1-y_0 \leq{}_{Q^\vee}\,\qp'(1)$.

 \par d) If $m_{\ql,\qm}(\qn)>0$ we have an Hecke path of shape $\ql$ (resp. $\qm$) from $0$ to $z'$ (resp. from $z'$ to $\qn$). So d) follows from b).
\end{proof}

 \begin{prop}\label{2.4} Suppose $(\qa_i^\vee)_{i\in I}$ free in $V$. Then for all $\ql,\qm,\qn\in Y^{++}$,  $m_{\ql,\qm}(\qn)$ is finite.
\end{prop}

\begin{NB} Actually  we may replace the condition $(\qa_i^\vee)_{i\in I}$ free by $(\qa_i^\vee)_{i\in I}$ $\R^+-$free. % (\ie $\sum a_i\qa_i^\vee=0,a_i\geq{}0\Rightarrow a_i=0,\forall i$).
 %This condition is actually equivalent to the existence of a $\qr\in V^*$ as in \ref{2.1}.
\end{NB}
\begin{proof} We have to count the $z\in\SHI_0^+$ such that $d^v(0,z)=\ql$ and $d^v(z,\qn)=\qm$.
 We set $z'=\qr_{-\infty}(z)$. By lemma \ref{2.3}b, $z'\in \ql-Q^\vee_+$ and $\qn\in z'+\qm-Q^\vee_+$, hence $z'$ is in $(\ql-Q^\vee_+)\cap(\qn-\qm+Q^\vee_+)$ which is finite as $(\qa_i^\vee)_{i\in I}$ is free or $\R^+-$free.
 So, we fix now $z'$. By \cite[cor. 5.9]{GR08} there is a finite number of Hecke paths $\qp'$ of shape $\qm$ from $z'$ to $\qn$. So, we fix now $\qp'$.
 And by [\lc th. 6.3] (see also \ref{4.10}, \ref{4.11}) there is a finite number of segments $[z,\qn]$ retracting to $\qp'$; hence the number of $z$ is finite.
\end{proof}

\begin{theo}\label{2.5} Suppose $(\qa_i^\vee)_{i\in I}$ free or $\R^+-$free, then $\shh_R$ is an algebra. %(\footnote{idem pour $\shh_R^{ba}$ ?})
\end{theo}

\begin{proof} We saw that for $\ql,\qm,\qn\in Y^{++}$,  $m_{\ql,\qm}(\qn)$ is finite; hence $c_\ql*c_\qm$ is well defined (eventually as an infinite formal sum).
 Let us consider $\qf,\psi\in\shh_R$:  $supp(\qf)\subset\cup_{i=1}^m\,(\ql_i-Q^\vee_+)$, $supp(\psi)\subset\cup_{j=1}^n\,(\qm_j-Q^\vee_+)$. Let $\qn\in Y^{++}$.
  If $m_{\ql,\qm}(\qn)>0$ with $\ql\in supp(\qf)$, $\qm\in supp(\psi)$ (hence $\ql\in \ql_i-Q^\vee_+$, $\qm\in \qm_j-Q^\vee_+$ for some $i,j$), we have $\ql+\qm\in\qn+Q^\vee_+$ by lemma \ref{2.3}d.
  So $\ql\in (\qn-\qm+Q^\vee_+)\cap(\ql_i-Q^\vee_+)\subset (\qn-\qm_j+Q^\vee_+)\cap(\ql_i-Q^\vee_+)$, a finite set.
   For the same  reasons $\qm$ is in a finite set, so $\qf*\psi$ is well defined.

   \par With the above notations $\qn\in(\ql+\qm-Q^\vee_+)\subset\cup_{i,j}\,(\ql_i+\qm_j-Q^\vee_+)$, so $\qf*\psi\in\shh_R$.
\end{proof}

\begin{defi}\label{2.6} $\shh_R=\shh_R(\SHI,G)$ is the {\it  spherical Hecke algebra} (with coefficients in $R$) associated to the hovel $\SHI$ and its strongly transitive automorphism group $G$.
\end{defi}

\begin{rema*} We shall now investigate $\shh_R$ and some other possible convolution algebras in $\widehat\shh_R$ by separating the cases: finite, indefinite and affine.
\end{rema*}

  \subsection{Finite case}\label{2.7} In this case $\QF$ and $W^v$ are finite, $(\qa_i^\vee)_{i\in I}$ is free, $\sht=V$ and the relation $\leq{}$ is trivial.
  The hovel $\SHI=\SHI^+$ is a locally finite Bruhat-Tits building.
  
%  \par For a triangle $[x,z,y]$ in $\SHI_0$ with $d^v(x,z)=\ql$ and $d^v(z,y)=\qm$, we have $d(x,y)\leq{}d(x,z)+d(z,y)=\Vert\ql\Vert+\Vert\qm\Vert$.
%  So, given $x$, the vertex $z$ (resp. $y$) has to be among the finite number of type$-0$ vertices in the ball of center $x$ and radius $\Vert\ql\Vert$ (resp. $\Vert\ql\Vert+\Vert\qm\Vert$).
%  It is now clear that, given $\ql$ and $\qm$, the number $m_{\ql,\qm}(\qn)$ is always finite and trivial for almost all $\qn$.

\par Let $\qr$ be the half sum of positive roots.
As $2\qr\in Q$ and $\qr(\qa_i^\vee)=1$, $\forall i\in I$, we see that an almost finite set in $Y^{++}$ is always finite. So $\shh_R$ and $\shh_R^f$ are equal.
% and also equal to $\shh_R^{ba}$ when $(\qa_i)_{i\in I}$ generates $V^*$.

\par The algebra  $\shh_\C$ was already studied by I. Satake in \cite{Sa63}. Its close link with buildings is explained in \cite{P06}.
The algebra $\shh_\Z$ is the spherical Hecke ring of \cite{KLM08}, where the interpretation of $m_{\ql,\qm}(\qn)$ as a number of triangles in $\SHI$ is already given.
%(\footnote{ Autres r\'ef\'erences plus appropri\'ees ?})

\par $\widehat\shh_R$ is not an algebra as \eg $m_{\ql,(-w_0)\ql}(0)\not=0$ $\forall\ql\in Y^{++}$ (where $w_0$ is the greatest element in $W^v$).
% A moment's thought tells that $\shh_R^{ba}$ is not an algebra when $Y\cap(\cap_{i\in I}\,Ker(\qa_i))\not=\{0\}$.
% But we may define completions of $\shh_R$ by asking that the supports are in a finite union of cones translated from a given sharp convex cone (???).

 \subsection{Indefinite case}\label{2.9}

 \begin{lemm*} Suppose now $\QF$ associated to an indefinite indecomposable generalized Cartan matrix. Then there is in $\QD^+_{im}$ an element $\qd$ (of support $I$) such that $\qd(\qa_i^\vee)<0$, $\forall i\in I$ and a basis $(\qd_i)_{i\in I}$ of the real vector space $Q_\R$ spanned by $\QF$  such that $\qd_i(\sht)\geq{}0$, $\forall i\in I$.
\end{lemm*}

\begin{proof} %By Hypothesis $\QD$ contains the root system $^\QF\QD$ generated by $\QF$. So we may apply \cite{K90} (or its generalization in \cite{Ba96}).
 Any $\qd\in\QD^+_{im}$ takes positive values on $\sht$ \cite[5.8]{K90}.
 Now, in the indefinite case, there is $\qd\in\QD^+_{im}\cap(\oplus_{i\in I}\,\R_{>0}.\qa_i)$ such that $\qd(\qa_i^\vee)<0$, $\forall i\in I$  [\lc 4.3], hence $\qd+\qa_i\in {\QD^+}$, $\forall i\in I$.
  Replacing eventually $\qd$ by $3\qd$ [\lc 5.5], we have $(\qd+\qa_i)(\qa_j^\vee)<0$, $\forall i,j\in I$, hence $\qd+\qa_i\in {\QD^+_{im}}$.
  The wanted basis is inside $\{\qd\}\cup\{\qd_0+\qa_i\mid i\in I\}$.
\end{proof}

%    \par To prove the existence of this wanted convolution product, we shall suppose here that $(\qa_i)_{i\in I}$ generates (\ie is a basis of) $V^*$. (\footnote{ On pourrait peut-\^etre utiliser le renseignement sur les valeurs de $\qr$ ("demi-somme des racines positives") donn\'e par [GR08; 5.8], mais cela ne semble pas clair et en plus dans \lc les chemins sont Hecke par rapport \`a $-C^v_f$ et on aura des Hecke par rapport \`a l'alcove $\g c_-$.})

    \par The existence of $\qd\in\QD^+_{im}$ as in the lemma proves that $(\qa_i^\vee)_{i\in I}$ is $\R^+-$free. So $\shh_R$ is an algebra.
    The following example \ref{2.10} proves that  $\shh_R^f$ is in general not a subalgebra.

    \par If  $(\qa_i)_{i\in I}$ generates (\ie is a basis of) $V^*$, $\widehat\shh_R$ is also an algebra (the {\it formal  spherical Hecke algebra}):
    Let $\qn\in Y^{++}$, we have to prove that there is only a finite number of pairs $(\ql,\qm)\in (Y^{++})^2$ such that $m_{\ql,\qm}(\qn)>0$.
    Let $z'$ be as in the proof of \ref{2.4}. We saw in \ref{2.2a} that $z'\in Y^+\cap(\qn-Y^+)=Y\cap\sht\cap(\qn-\sht)$.
    By the lemma, $\sht\cap(\qn-\sht)$ is bounded, hence $Y\cap\sht\cap(\qn-\sht)$ is finite.
    So we may fix $z'$. Now $\ql\in z'+Q^\vee_+$ hence (for $\qd$ as in the lemma) $\qd(\ql)\leq{}\qd(z')$; as $\qa_i(\ql)\in\Z_{>0}$ $\forall i\in I$ and $\qd\in \oplus_{i\in I}\,\R_{>0}.\qa_i$ this gives only a finite number of possibilities for $\ql$.
    Similarly $\qm\in\qn-z'+Q_+^\vee$ has to be in a finite set.

%    \par Let $y\in \SHI_0^+$,  and $w\in W^v$ be such that there exists $z\in \SHI_0^+$ with $d^v(z,y)=\qm$ and $\qr_-(z)=w.\ql\in Y^+$.
 %   Then $w.\ql\in Y^+\subset \sht$ and $w.\ql\leq{}y$, as $\qr_-$ is increasing, \cf \cite[2.8]{R11} in a very similar situation. (\footnote{\'ecrire la d\'emonstration dans ce cas ??})
%    So $w.\ql\in\sht\cap(y-\sht)$. But, by lemma \ref{2.6}, $\sht$ is a strictly convex cone, so $\sht\cap(y-\sht)$ is bounded and, as $Y$ is discrete, $Y\cap\sht\cap(y-\sht)$ is finite.
 %    Thus the sums $m_{\ql,\qm}(y) =\sum_{w\in W^v/(W^v)_\lambda}\,N(\qm,w.\ql,y)$ and $\sum_{\ql,\qm\in Y^{++}}\,m_{\ql,\qm}(y)$ are finite.
  %   We have proved that the convolution product $\qf*\psi$ is well defined on $\widehat\shh_R$.

\par Actually  $\widehat\shh_R$ is often equal to $\shh_R$ when $(\qa_i^\vee)_{i\in I}$ is free and $(\qa_i)_{i\in I}$ generates  $V^*$ (hence the matrix $\M=(\qa_j(\qa_i^\vee))$ is invertible), see the following example \ref{2.10}.

\subsection{An indefinite rank $2$ example}\label{2.10} Let us consider the Kac-Moody matrix $\M=\begin{pmatrix}2 & -3\\-3 & 2\\\end{pmatrix}$. The basis of $\QF$ and $V^*$ is $\{\qa_1,\qa_2\}$ and we consider the dual basis $(\varpi_1^\vee,\varpi_2^\vee)$ of $V$.
 In this basis $\qa_1^\vee=\begin{pmatrix}2\\-3\\\end{pmatrix}$, $\qa_2^\vee=\begin{pmatrix}-3\\2\\\end{pmatrix}$ and the matrices of $r_1$, $r_2$, $r_2r_1$ and $r_1r_2$ are respectively $\begin{pmatrix}-1 & 0\\3 & 1\\\end{pmatrix}$, $\begin{pmatrix}1 & 3\\0 & -1\\\end{pmatrix}$, $M=\begin{pmatrix}8 & 3\\-3 & -1\\\end{pmatrix}$ and $M^{-1}=\begin{pmatrix}-1 & -3\\3 & 8\\\end{pmatrix}$.
 The eigenvalues of $M$ or $M^{-1}$ are $a_\pm{}=(7\pm{}\sqrt{45})/2$.
  In a basis diagonalizing $M$ and $M^{-1}$ we see easily that $(r_2r_1)^n+(r_1r_2)^n=a_n.Id_V$ where $a_n=a_+^n+a_-^n$ is in $\N$ and increasing up to infinity ($a_0=2$, $a_1=7$, $a_2=47$, $a_3=322$,...).

  \par Consider now $\ql=\qm=-\qa^\vee_1-\qa^\vee_2=\begin{pmatrix}1\\1\\\end{pmatrix}$ in $Y^{++}\subset\Z_{\geq{}0}.\varpi^\vee_1\oplus \Z_{\geq{}0}.\varpi^\vee_2$.
  % (we suppose $\QL_\qa=\Z$, $\forall\qa\in\QF$).
  We have $(r_2r_1)^n.\ql+(r_1r_2)^n.\ql=a_n.\ql$. This means that $m_{\ql,\ql}(a_n.\ql)\geq{}N_{\g c_-}(\ql,(r_2r_1)^n\ql,a_n.\ql)\geq{}1$, for all positive $n$ (and the same thing for $N_{-\infty}$).
  So $c_\ql*c_\ql$ is an infinite formal sum.

  \par Actually $(-Q_+^\vee)\cap Y^{++}\supset\Z_{\geq{}0}.5\varpi^\vee_1\oplus \Z_{\geq{}0}.5\varpi^\vee_2$, hence $Y^{++}$ itself is almost finite!

\subsection{An affine rank $2$ example}\label{2.11} Let us consider the Kac-Moody matrix $\M=\begin{pmatrix}2 & -2\\-2 & 2\\\end{pmatrix}$.
The basis of $\QF$ is $\{\qa_1,\qa_2\}$ but we consider a realization $V$ of dimension $3$  for which $\{\qa_1^\vee,\qa_2^\vee\}$ is free and with basis of $V^*$ $\{\qa_o=-\qr,\qa_1,\qa_2\}$.
 More precisely, if $(\varpi_0^\vee,\varpi_1^\vee,\varpi_2^\vee)$ is  the dual basis of $V$, we have  $\qa_1^\vee=\begin{pmatrix}-1\\2\\-2\\\end{pmatrix}$, $\qa_2^\vee=\begin{pmatrix}-1\\-2\\2\\\end{pmatrix}$ and the matrices of $r_1$, $r_2$, $r_1r_2$ and $r_2r_1$ are respectively $\begin{pmatrix}1&1&0\\0&-1 & 0\\0&2 & 1\\\end{pmatrix}$, $\begin{pmatrix}1&0&1\\0&1 & 2\\0&0 & -1\\\end{pmatrix}$, $M=\begin{pmatrix}1&1&3\\0&-1 & -2\\0&2 & 3\\\end{pmatrix}$ and $M^{-1}=\begin{pmatrix}1&3&1\\0&3 & 2\\0&-2 & -1\\\end{pmatrix}$.
 A classical calculus using triangulation tells us that  $(r_2r_1)^n+(r_1r_2)^n=\begin{pmatrix}2&4n^2&4n^2\\0&1 & 0\\0&0 & 1\\\end{pmatrix}$.
 Actually $c=\qa_1^\vee+\qa_2^\vee=-2\varpi_0^\vee\in Q_+^\vee$ is the canonical central element \cite[{\S{}} 6.2]{K90} and the above calculations are peculiar cases of [\lc {\S{}} 6.5].

\par Let's consider now $\ql=\qm=\sum^2_{i=1}\,a_i\varpi_i^\vee\in Y^{++}\subset \oplus^2_{i=1}\,\Z_{\geq{}0}\varpi_i^\vee$.
% (we suppose $\QL_\qa=\Z$ $\forall\qa\in\QF$).
 We have $(r_2r_1)^n(\ql)+(r_1r_2)^n(\ql)=\ql-2n^2\vert\ql\vert c$ with $\vert\ql\vert=a_1+a_2$.
 This means that $m_{\ql,\ql}(\ql-2n^2\vert\ql\vert c)\geq{}N_{\g c_-}(\ql,(r_2r_1)^n(\ql),\ql-2n^2\vert\ql\vert c)\geq{}1$, $\forall n\in \Z$ (and the same thing for $N_{-\infty}$).
 So $c_\ql*c_\ql$ is an infinite formal sum.

 \par Moreover as $c$ is fixed by $r_1$ and $r_2$, $(r_2r_1)^n(\ql+2n^2\vert\ql\vert c)+(r_1r_2)^n(\ql)=\ql$,  so $m_{\ql+2n^2\vert\ql\vert c,\ql}(\ql)\geq{}1$, $\forall n\in \Z$, and $\widehat\shh_R$ is not an algebra.

 \par Remark also that, if we consider the essential quotient $V^e=V/\R c$, the above calculus tells that $m_{\ql,\ql}(\ql)\geq{}\sum_{n\in\Z}\,N_{\g c_-}(\ql,(r_2r_1)^n(\ql),\ql)$ is infinite if $\vert\ql\vert>0$.

       \subsection{Affine indecomposable case}\label{2.12}
       %Le probl\`eme ici est que je n'arrive pas \`a faire marcher la machine. En tout cas $m_{\ql,\qm}(\qn)=\sum_{w\in W^v}\,N(\qm,w.\ql,\qn)$ est infini d\`es que $\qn=\ql+\qm$ ou m\^eme $\qn\in W^v.\ql+W^v.\qm$ parce qu'il y a alors une infinit\'e de $N(\qm,w\ql,\qn)$  non nuls (quand $w$ varie dans un sous-groupe distingu\'e d'indice fini $\vect T$ de $W^v$). J'esp\'erais m'en tirer en montrant que ces  $N(\qm,w\ql,\qn)$ sont alors \'egaux et en ne comptant qu'une fois tous ces nombre \'egaux c'est \`a dire en sommant sur $W^v/\vect T$. Mais en fait cela ne semble pas marcher tout le temps.
\par We saw in the example \ref{2.11} above that $m_{\ql,\ql}(\ql)$ may be infinite, $\forall\ql\in Y^{++}$ when $(\qa_i^\vee)_{i\in I}$ is not free. So, in this case, $\widehat\shh_R$ seems to contain no algebra except $R.c_0$.

\par Remark also that $(\qa_i^\vee)_{i\in I}$ free is equivalent to $(\qa_i^\vee)_{i\in I}$ $\R^+-$free in the affine indecomposable case as the only possible relation between the $\qa_i^\vee$ is $c=0$ where $c=\sum_{i\in I}\,a_i^\vee.\qa_i^\vee$ (with $a_i^\vee\in\Z_{>0}$ $\forall i\in I$) is the canonical central element.

\par An almost finite subset in $Y^{++}$ is a finite union of subsets like $Y_\ql=(\ql-Q_+^\vee)\cap Y^{++}$.
 Let $\qd$ be the smallest positive imaginary root in $\QD$. Then $\qd(Q_+^\vee)=0$ so $Y_\ql\subset\{y\in Y^{++}\mid\qd(y)=\qd(\ql)\}=Y_\ql'$.
  But $\qd=\sum_{i\in I}\,a_i.\qa_i$ with $a_i\in\Z_{>0}$ $\forall i\in I$, so the image of $Y'_\ql$ in $V^e=V/\R c$ (where $\R c=\cap_{i\in I}\,Ker(\qa_i)$) is finite.
  It is now clear that $Y_\ql$ is a finite union of sets like $\qm-\Z_{\geq{}0}.c$ with $\qm\in Y^{++}$.
   Hence an almost finite subset as defined above is the same as an almost finite union (of double cosets) as defined in \cite{BrK10}.

   \par The algebra $\shh_\C$ is the one introduced by A. Braverman and D. Kazhdan in \cite{BrK10}. We gave above a combinatorial proof that it is an algebra, without algebraic geometry.

%%%%%%%%%%%%%%%%%%%%%%%%%%%%%%%%%%%%%%%%%%%%
\section{The split Kac-Moody case}\label{s3}

\subsection{Situation}\label{3.1} As in \cite{R12} or \cite{R13}, we consider a split Kac-Moody group $\g G$ %(\footnote{generaliser au cas almost split ?}) 
associated to a root generating system (RGS) $\shs=(\M,Y_\shs,(\overline\qa_i)_{i\in I}, (\qa_i^\vee)_{i\in I})$ over a field $\shk$ endowed  with a discrete valuation $\qo$ (with value group $\QL=\Z$ and ring of integers $\sho=\qo^{-1}([0,+\infty])$) whose residue field $\qk=\F_q$ is finite .
 So, $\M=(a_{i,j})_{i,j\in I}$ is a Kac-Moody matrix, $Y_\shs$ a free $\Z-$module, $(\qa_i^\vee)_{i\in I}$ a family in $Y_\shs$, $(\overline\qa_i)_{i\in I}$ a family in the dual $X=Y_\shs^*$ of $Y_\shs$ and $\overline\qa_j(\qa_i^\vee)=a_{i,j}$.

 \par If $(\overline\qa_i)_{i\in I}$ is free in $X$, we consider $V=V_Y=Y_\shs\otimes_\Z\R$ and the clear quadruple $(V,W^v,(\qa_i=\overline\qa_i)_{i\in I}, (\qa^\vee_i)_{i\in I})$.
  In general, we may define $Q=\Z^I$ with canonical basis $(\qa_i)_{i\in I}$, then $V=V_Q=Hom_\Z(Q,\R)$ is also in a quadruple as in \ref{1.1}.
   A third example $V^ {xl}$ of choice for $V$ is explained in \cite{R13}.
   We always denote by $bar:Q\to X$ the linear map sending $\qa_i$ to $\overline\qa_i$.

   \par With these vectorial data we may define what was considered in \ref{1.1} and \ref{1.2} (we choose $\QL_\qa=\QL=\Z$, $\forall\qa\in\QF$).

   \par Now the hovel $\SHI$ in \ref{1.3} is as defined in \cite{R12} or \cite{R13} and the strongly transitive group is $G=\g G(\shk)$. By \cite[6.11]{R11} or \cite[5.16]{R12} we have $q_M=q$ for any wall $M$.
   
   \par When $\g G$ is a split reductive group, $\SHI$ is its extended Bruhat-Tits building.

   \subsection{Generators for $G$}\label{3.2} The Kac-Moody group $\g G$ contains a split maximal torus $\g T$ with character group $X$ and cocharacter group $Y_\shs$. We note $T=\g T(\shk)$.
   For each $\qa\in\QF\subset Q$ there is a group homomorphism $x_\qa:\shk\to G$ which is one-to-one; its image is the subgroup $U_\qa$.
   Now $G$ is generated by $T$ and the subgroups $U_\qa$ for $\qa\in\QF$, submitted to some relations given by Tits \cite{T87}, also available in \cite{Re02} or \cite{R12}.
   We set $U^\pm{}$ the subgroup generated by the subgroups $U_\qa$ for $\qa\in\QF^\pm{}$.

   \par We shall explain now only a few of the relations. For $u\in\shk$, $t\in T$ and $\qa\in\QF$ one has:

   \par (KMT4) $t.x_\qa(u).t^{-1}=x_\qa(\overline\qa(t).u)$\qquad(where $\overline\qa=bar(\qa)$)

   \par\noindent For $u\not=0$, we note $\widetilde s_\qa(u)=x_\qa(u).x_{-\qa}(u^{-1}).x_\qa(u)$ and $\widetilde s_\qa=\widetilde s_\qa(1)$.

   \par (KMT5) $\widetilde s_\qa(u).t.\widetilde s_\qa(u)^{-1}=r_\qa(t)$\qquad($W^v$ acts on $V$, $Y_\shs$, $X$ hence on $T$)

   \subsection{Weyl groups}\label{3.3} Actually the stabilizer $N$ of $\A\subset\SHI$ is the normalizer of $\g T$ in $G$.
   The image $\qn(N)$ of $N$ in $Aut(\A)$ is a semi-direct product $\qn(N)=\qn(N_0)\ltimes\qn(T)$ with:

   \par $N_0$ is the fixator of $0$ in $N$ and $\qn(N_0)$ is isomorphic to $W^v$ acting linearly on $\A=V$. Actually $\qn(N_0)$ is generated by the elements $\qn(\widetilde s_\qa)$ which act as $r_\qa$ (for $\qa\in\QF$).

   \par $t\in T$ acts on $\A$ by a translation of vector $\qn(t)\in V$ such that $\overline\chi(\qn(t))=-\qo(\chi(t))$ for any $\overline\chi\in X=Y_\shs^*$ and $\chi\in X$ or $Q$ which are related by $\overline\chi=\chi$ if $V=V_Y$ or $\overline\chi=bar(\chi)$ if $V=V_Q$.

   \par So, $\qn(N)=W^v\ltimes Y$ where $Y$ is closely related to $Y_\shs\simeq T/\g T(\sho)$: as $\QL=\qo(\shk)=\Z$, they are equal if $V=V_Y$ and, if $V=V_Q$, $Y=bar^*(Y_\shs)$ is the image of $Y_\shs$ by the map $bar^*:Y_\shs\to Hom_\Z(Q,\Z)$ dual to $bar$.

   \par So, the choice $V=V_Y$ is more pleasant. The choice $V=V_Q$ is made \eg in \cite{Ch10}, \cite{Ch11} or \cite{Re02} and has good properties in the indefinite case, \cf \ref{2.9}.
   They coincide both when $(\overline\qa_i)_{i\in I}$ is a basis of $X\otimes\R=V_Y^*$. This assumption generalizes semi-simplicity, in particular the center of $\g G$ is then finite \cite[9.6.2]{Re02}.

   \subsection{The group $K$}\label{3.4} The group $K=G_0$ should be equal to $\g G(\sho)$ for some integral structure of $\g G$ over $\sho$ \cf \cite[3.14]{GR08}.
    But the appropriate integral structure is difficult to define in general. So we define $K$ by its generators:

    \par The group $N_0$ is generated by $T_0=\g T(\sho)=T\cap K$ and the elements $\widetilde s_\qa$ for $\qa\in\QF$ (this is clear by \ref{3.3}).
    The group $U_0$, generated by the groups $U_{\qa,0}=x_\qa(\sho)$ for $\qa\in\QF$, is in $K$.
    We note $U_0^\pm{}=U_0\cap U^\pm{}$. In general $U_0^\pm{}$ is not generated by the groups $U_{\qa,0}$ for $\qa\in\QF^\pm{}$ \cite[4.12.3a]{R12}.

    \par It is likely that $K$ may be greater than the group generated by $N_0$ and $U_0$ (\ie by $U_0$ and $T_0$).
    We have to define groups $U_0^{pm+}\supset U_0^+$ and $U_0^{nm-}\supset U_0^-$ as follows. In a formal positive completion $\widehat G^+$ of $G$, we can define a subgroup
    $U_0^{ma+}=\prod_{\qa\in\QD^+}\,U_{\qa,0}$ of  the subgroup $U^{ma+}=\prod_{\qa\in\QD^+}\,U_{\qa}$ of  $\widehat G^+$, with $U^+\subset  U^{ma+}$
     (where $U_{\qa,0}$ and $U_\qa$ are suitably defined for $\qa$ imaginary). Then $U_0^{pm+}=U_0^{ma+}\cap G=U_0^{ma+}\cap U^+$.
    The group $U_0^{nm-}$ is defined similarly with $\QD^-$ using a group $U_0^{ma-}\subset U^{ma-}$ in a formal negative completion $\widehat G^-$ of $G$.

    \par Now $K=G_0=U_0^{nm-}.U_0^+.N_0=U_0^{pm+}.U_0^-.N_0$\qquad\cite[4.14, 5.1]{R12}

\begin{rema*}
Let us denote by $K_1$ the group used by A. Braverman, D. Kazhdan and M. Patnaik in their definition of the spherical Hecke algebra. 
With the notation above, $K_1$ is generated by $T_0$ and $U_0$, \ie by $T_0$, $U_0^+$ and $U_0^-$, hence $K=U_0^{nm-}.K_1=U_0^{pm+}.K_1$, with $U_0^-\subset U_0^{nm-}\subset U^-$ and $U_0^+\subset U_0^{pm+}\subset U^+$.
  But they prove, at least in the untwisted affine case, that $U^-\cap U^+.K_1\subset K_1$ \cite[proof of 6.4.3]{BrKP12}; so $U_0^{nm-}\subset U^-\cap K\subset U^-\cap U^+.K_1\subset K_1$ and $K=K_1$.
  This result answers positively a question in \cite[5.4]{R13}, at least for points of type $0$ and in the untwisted affine split case.
\end{rema*}

    \begin{prop}\label{3.5} There is an involution $\theta$ (called Chevalley involution) of the group $G$ such that $\theta(t)=t^{-1}$ for all  $ t\in T$ and $\theta(x_\qa(u))=x_{-\qa}(u)$ for all $\qa\in\QF$ and $u\in\shk$.
     Moreover $K$ is $\theta-$stable and $\theta$ induces the identity on $W^v=N/T$.
    \end{prop}

    \begin{proof} This involution is well known on the corresponding complex Lie algebra, see \cite[1.3.4]{K90} where one uses for the generators $e_\qa$ a convention different from ours ($[e_\qa,e_{-\qa}]=-\qa^\vee$ as in \cite{T87} or \cite{Re02}).
    Hence the proposition follows when $\qk$ contains $\C$ or is at least of characteristic $0$.
    But here we have to use the definition of $G$ by generators and relations.

    \par We see in \cite[1.5, 1.7.5]{R12} that $\widetilde s_{\qa}(-u)=\widetilde s_{\qa}(u)^{-1}$ and $\widetilde s_{\qa}(u)=\widetilde s_{-\qa}(u^{-1})$.
     So for the wanted involution $\theta$ we have $\theta(\widetilde s_{\qa}(u))=\widetilde s_{-\qa}(u)=\widetilde s_{\qa}(u^{-1})$. We have now to verify the relations between the $\theta(x_\qa(u))=x_{-\qa}(u)$, $\theta(t)=t^{-1}$ and $\theta(\widetilde s_{\qa}(u))=\widetilde s_{\qa}(u^{-1})$.
     This is clear for (KMT4) and (KMT5) (as $r_\qa=r_{-\qa}$). The three other relations are:

  \par (KMT3) $(x_\qa(u),x_\qb(v))=\prod x_\qg(C^{\qa,\qb}_{p,q}.u^pv^q)$ for $(\qa,\qb)\in\QF^2$ prenilpotent and, for the product, $\qg=p\qa+q\qb$ runs in $(\Z_{>0}\qa+\Z_{>0}\qb)\cap\QF$.
  But the integers $C^{\qa,\qb}_{p,q}$ are picked up from the corresponding formula between exponentials in the automorphism group of the corresponding complex Lie algebra.
  As we know that $\theta$ is defined in this Lie algebra, we have $C^{-\qa,-\qb}_{p,q}=C^{\qa,\qb}_{p,q}$ and (KMT3) is still true for the images by $\theta$.

  \par (KMT6) $\widetilde s_{\qa}(u^{-1})=\widetilde s_{\qa}.\qa^\vee(u)$ for $\qa$ simple and $u\in\shk\setminus\{0\}$.
  \par\noindent This is still true after a change by $\theta$ as $\theta(\widetilde s_{\qa}(u^{-1}))=\widetilde s_{\qa}(u)$ and $(-\qa)^\vee(u)=\qa^\vee(u^{-1})$.

  \par (KMT7) $\widetilde s_{\qa}.x_\qb(u).\widetilde s_{\qa}^{-1}=x_\qg(\qe.u)$ if $\qg=r_\qa(\qb)$ and $\widetilde s_{\qa}(e_\qb)=\qe.e_\qg$ in the Lie algebra (with $\qe=\pm{}1$).
   This is still true after a change by $\theta$ because $\widetilde s_{\qa}(e_\qb)=\qe.e_\qg\Rightarrow\widetilde s_{\qa}(e_{-\qb})=\qe.e_{-\qg}$ (as $r_\qa(\qb^\vee)=\qg^\vee$).

   \par So, $\theta$ is a well defined involution of $G$,  $\theta(U_0)=U_0$, $\qth(N_0)=N_0$ and $\qth(U_0^\pm{})=U_0^\mp$.
   But the isomorphism $\qth$ of $U^+$ onto $U^-$ can clearly be extended to an isomorphism $\qth$ from $U^{ma+}$ onto $U^{ma-}$ sending $U_0^{ma+}$ onto $U_0^{ma-}$.
   So $\qth(U_0^{pm+})=U_0^{nm-}$ and $\qth(K)=K$.
   As $\qth(\widetilde s_{\qa})=\widetilde s_{\qa}$, $\qth$ induces the identity on $W^v=N/T$.
    \end{proof}

\begin{theo}\label{3.6} The algebra $\widehat\shh_R$ or $\shh_R$ is commutative, when it exists.
\end{theo}

    \par\noindent{\bf Notation}: To be clearer we shall sometimes write $\widehat\shh_R(\g G,\shk)$ or $\shh_R(\g G,\shk)$ instead of $\widehat\shh_R$ or $\shh_R$.

\begin{proof} The formula $\qth^\#(g)=\qth(g^{-1})$ defines an anti-involution ($\qth^\#(gh)=\qth^\#(h).\qth^\#(g)$ ) of $G$ which induces the identity on $T$ and stabilizes $K$.
 In particular $\qth^\#(G^+)=\qth^\#(KY^{++}K)=G^+$ and $\qth^\#(K\ql K)=K\ql K$, $\forall\ql\in Y^{++}$.
  For $\qf,\psi\in\widehat\shh_R$ and $g\in G^+$, one has:\quad $(\qf*\psi)(g)=$\goodbreak\noindent$(\qf*\psi)(\qth^\#(g))=\sum_{h\in G^+/K}\:\qf(h)\psi(h^{-1}\qth^\#(g))$.
  The map $h\mapsto h'=\qth^\#(h^{-1}\qth^\#(g))=g\qth^\#(h^{-1})$ is one-to-one from $G^+/K$ onto $G^+/K$. So, $(\qf*\psi)(g)=\sum_{h'\in G^+/K}\:\qf(\qth^\#(h'^{-1}g))\psi(\qth^\#(h'))=\sum_{h'\in G^+/K}\:\qf(h'^{-1}g)\psi(h')=(\psi*\qf)(g)$.
\end{proof}

\begin{remas}\label{3.7} 1) This commutativity will be below proved in general as a consequence of the Satake isomorphism.
 The above proof generalizes well known proofs in the reductive case, \eg for $\g G=\g{GL}_n$,  $\qth^\#$ is the transposition.

2) When $\g G$ is an almost split Kac-Moody group over the field $\shk$ (supposed complete or henselian) it splits over a finite Galois extension $\shl$, the hovel $^\shk\!\SHI$ over $\shk$ exists and embeds in the hovel $^\shl\!\SHI$ over $\shl$ \cite[{\S{}} 6]{R13}.
 After enlarging eventually $\shl$ one may suppose that $0$ is a special point in $^\shk\!\SHI$ and $^\shl\!\SHI$, more precisely in the fundamental apartments $^\shk\!\A\subset {^\shl\!\A}=\A$ associated respectively to a maximal $\shk-$split torus $_\shk\g S$ and a $\shl-$split maximal torus $\g T\supset{_\shk\g S}$.
  If we make a good choice of the homomorphisms $x_\qa:\shl\to\g G(\shl)$, the associated involution $\theta$ of $\g G(\shl)$ should commute with the action of the Galois group $\QG=Gal(\shl/\shk)$ hence induce an involution $^\shk\theta$ and an anti-involution $^\shk\theta^\#$ of $\g G(\shk)=\g G(\shl)^\QG$ such that $^\shk\theta(K)={^\shk\theta^\#}(K)=K$ and $^\shk\theta^\#$ induces the identity in $Y({_\shk\g S})={_\shk\g S}(K)/{_\shk\g S}(\sho)$.
  The commutativity of $\widehat\shh_R(\g G,\shk)$ or $\shh_R(\g G,\shk)$ would follow.
  
  \par This strategy works well when $\g G$ is quasi split over $\shk$; unfortunately it seems to fail in the general case.
 
% Suppose that $^\shk\!\SHI$ intersects $\SHI_0$, then we may choose $0\in{^\shk\!\SHI}$ and the fundamental apartments $^\shk\A\subset\A$ with the same origin $0$.
% Suppose moreover $\A$ stable under $\QG=Gal(\shl/\shk)$ \ie the corresponding torus defined over $\shk$.
%  Then it seems clear that $\qth$ commutes to $\QG$, so its restriction $^\shk\qth$ to $\g G(\shk)=\g G(\shl)^\QG$ has good properties and the above properties tell that $\widehat\shh_R(\g G,\shk)$ or $\shh_R(\g G,\shk)$ is commutative. 
  %(\footnote{id\'ee \`a creuser: en choisissant $\shl$ assez grand et assez ramifi\'e, on peut sans doute traiter le cas presque d\'eploy\'e g\'en\'eral.})

  \par 3) The commutativity of $\widehat\shh_R$ or $\shh_R$ is linked to the choice of a special vertex for the origin $0$.
  Even in the semi-simple case, other choices may give non commutative convolution algebras, see \cite{Sa63} and \cite{KeR07}.

\end{remas}
%%%%%%%%%%%%%%%%%%%%%%%%%%%%%%%%%%%%%%%%%%%%
\section{Structure constants}\label{s4}

We come back to the general framework of {\S{}} \ref{s1}. We shall compute the structure constants of $\widehat\shh_R$ or $\shh_R$ by formulas depending on $\A$ and the numbers $q_M$ of \ref{1.3}.
 Note that there are only a finite number of them: as $q_{wM}=q_M$, $\forall w\in\qn(N)$ and $wM(\qa,k)=M(w\qa,k),\forall w\in W^v$, we may suppose $M=M(\qa_i,k)$ with $i\in I$ and $k\in \Z$.
  Now  $\qa_i^\vee\in Q^\vee\subset Y$; as $\qa_i(\qa_i^\vee)=2$ the translation by $\qa_i^\vee$ permutes the walls $M=M(\qa_i,k)$ (for  $k\in \Z$) with two orbits.
  So $Y$ has at most two orbits in the set of the constants $q_{M(\qa_i,k)}$, those of $q_i=q_{M(\qa_i,0)}$ and $q_i'=q_{M(\qa_i,\pm{}1)}$.
  Hence the number of (possibly) different parameters is at most $2.\vert I\vert$. We denote by $\shq=\{q_1,\cdots,q_l,q'_1=q_{l+1},\cdots,q'_l=q_{2l}\}$ this set of parameters.

\subsection{Centrifugally folded galleries of chambers}
\label{4.1}

Let $x$ be a point in the standard apartment $\mathbb A$. Let $\Phi_x$ be the set of all roots $\alpha$ such that $\alpha (x)\in \Z$. It is a closed subsystem of roots. Its associated Weyl group $W^v_x$ is a Coxeter group.
% with canonical generators $r_\alpha$, $\alpha$ simple in $\Phi_x^+$.

\par We have twinned buildings $\SHI^+_x$ (resp. $\SHI^-_x$) whose elements are segment germs $[x,y)=germ_x([x,y])$ for $y\in\SHI$, $y\not=x$, $y\geq{}x$ (resp. $y\leq{}x$).
 We consider their unrestricted structure, so the associated Weyl group is $W^v$ and the chambers (resp. closed chambers) are the local chambers $C=germ_x(x+C^v)$ (resp. local closed chambers $\overline{C}=germ_x(x+\overline{C^v})$), where $C^v$ is a vectorial chamber, \cf \cite[4.5]{GR08} or \cite[{\S{}} 5]{R11}.
  To $\A$ is associated a twin system of apartments $\A_x = (\A_x^-,\A_x^+)$.

  \par  We choose in $\A^-_x$ a negative (local) chamber $C^-_x$ and denote $C^+_x$ its opposite in $\A^+_x$.
  We consider the system of positive roots $\QF^+$ associated to $C^+_x$ (\ie $\QF^+=w\QF^+_f$, if $\QF^+_f$ is the system $\QF^+$ defined in \ref{1.1} and $C^+_x=germ_x(x+wC^v_f)$).
   We note $(\qa_i)_{i\in I}$ the corresponding basis of $\QF$ and $(r_i)_{i\in I}$ the corresponding generators of $W^v$.

\par Fix a reduced decomposition of  an element $w\in W^v$, $w = r_{i_1}\cdots r_{i_r}$ and let
$\mathbf i = (i_1,..., i_r)$ be the type of the decomposition.
%We denote by $\alpha_{i_j}$ the simple root in $\Phi$ corresponding to $r_{i_j}$.
%Recall that we have  of type $W_x^v$.
 We consider now galleries of (local) chambers $\mathbf c = (C_x^-,C_1,...,C_r)$ in the
apartment $\mathbb A_x^-$ starting at $C_x^-$ and of type $\mathbf i$. The set of all these galleries is in bijection with the set $\Gamma (\mathbf i) = \{1,r_{i_1}\}\times\cdots\times \{1,r_{i_r}\}$ via the map $(c_1,...,c_r)\mapsto (C_x^-, c_1 C_x^-,...,c_1\cdots c_r C_x^-)$.
Let $\beta_j = -c_1\cdots c_j (\alpha_{i_j})$, then $\beta_j$ is the root corresponding to the common
limit hyperplane $M_j = M_{\beta_j}$ of $C_{j-1} =c_1\cdots c_{j-1} C_x^- $ and $C_j = c_1\cdots c_j C_x^-$ and satisfying to $\qb_j(C_j)\geq{}\qb_j(x)$ (actually $M_j$ is a wall $\iff\qb_j\in\QF_x$).
In the following, we shall identify a sequence $(c_1,...,c_r)$ and the corresponding gallery.

\begin{defi}\label{4.2}
Let $\mathfrak Q$ be a chamber in $\mathbb A_x^+$. A gallery $\mathbf c = (c_1,...,c_r)\in\Gamma(\mathbf i)$ is said to be centrifugally folded with
respect to $\mathfrak Q$ if $c_j = 1$ implies $\qb_j\in\QF_x$ and $w_{\mathfrak Q}^{-1} \beta_j < 0$, where $w_{\mathfrak Q} = w(C_x^+, \mathfrak Q)\in W^v$ (\ie $\g Q=w_{\g Q}C^+_x$).
We denote this set of centrifugally folded galleries by $\Gamma_{\mathfrak Q}^+(\mathbf i)$.
\end{defi}

\begin{prop}\label{4.3}
A gallery $\mathbf c = (C_x^-,C_1,...,C_r)\in\Gamma(\mathbf i)$ belongs to $\Gamma_{\mathfrak Q}^+(\mathbf i)$ if, and only if, $C_j = C_{j-1}$ implies that  $M_j = M_{\beta_j}$ is a wall and separates $\mathfrak Q$ from $C_j = C_{j-1}$.
\end{prop}

\noindent{\it Proof. } We saw that $M_j$ is a wall $\iff\qb_j\in\QF_x$. We have the following equivalences:

\noindent ($M_j$ separates $\mathfrak Q$ from $C_j = C_{j-1}$) $\Longleftrightarrow$ ($w_{\mathfrak Q}^{-1}M_j$
separates $C_x^+$ from $w_{\mathfrak Q}^{-1}C_j = w_{\mathfrak Q}^{-1}C_{j-1}$)  $\Longleftrightarrow$ ($w_{\mathfrak Q}^{-1} \beta_j $ is a negative root).
\qed

\bigskip
%Now we consider the residue building $\SHI_x$ endowed with its restricted structure of twin building.
\par The group $\overline G_x = G_x/G_{\SHI_x}$ acts strongly transitively on $\SHI_x^+$ and $\SHI_x^-$.
For any root $\alpha\in \Phi_x$ with $\alpha(x) = k\in\Z$, the group $\overline U_\alpha = U_{\alpha, k}/U_{\alpha, k +1}$ is a finite subgroup of $\overline G_x$ of cardinality $q_{x,\qa}= q_{M(\alpha, -\alpha(x))}\in\shq$. We denote by $u_\alpha$ the elements of this group.

\medskip

Next, let $\qr_{\mathfrak Q} : \SHI_x \to \mathbb A_x$ be the retraction centered at $\mathfrak Q$. To a gallery of chambers $\mathbf c = (c_1,...,c_r) = (C_x^-,C_1,...,C_r)$ in $\Gamma(\mathbf i)$,
one can associate the set of all galleries of type $\mathbf i$ starting at $C_x^-$ in $\SHI_x^-$ that retract onto $\mathbf c$, we denote this set by $\mathcal C_{\mathfrak Q}(\mathbf c)$.
We denote the set of minimal galleries in $\mathcal C_{\mathfrak Q}(\mathbf c)$ by $\mathcal C^m_{\mathfrak Q}(\mathbf c)$.
Set
\begin{equation}
\label{eq:Gal}
g_j =
\left\{
\begin{array}{ll}
c_j & \hbox{ if } w_{\mathfrak Q}^{-1} \beta_j > 0  \hbox{ or } \qb_j\not\in\QF_x \\
u_{c_j(\alpha_{i_j})} c_j & \hbox{ if } w_{\mathfrak Q}^{-1} \beta_j < 0 \hbox{ and } \qb_j\in\QF_x.
\end{array}\right.
\end{equation}

%The minimal galleries in $\mathcal C_{\mathfrak Q}(\mathbf c)$
%are those such that for any $j$: $C'_{j-1}\ne C'_j$, i.e. $c_j\not=1$ if $w_{\mathfrak Q}^{-1} \beta_j > 0$,
%and $u_{c_j(\alpha_{i_j})}\ne 1$ if $c_j=1$ and $w_{\mathfrak Q}^{-1} \beta_j < 0$.

\begin{prop}\label{4.4} $\mathcal C_{\mathfrak Q}(\mathbf c)$ is the non empty set of all galleries $(C_x^-=C'_0,C'_1,...,C'_r)$ where $\forall j: \ C'_j = g_1\cdots g_j C_x^- $ with each $g_j$ chosen as in (\ref{eq:Gal}) above.
For all $j$ the local chambers $\mathfrak Q$ and $C'_j$ are in the apartment $g_1\cdots g_j\A_x$.

\par The set $\mathcal C^m_{\mathfrak Q}(\mathbf c)$ is empty if, and only if, the gallery $\mathbf c$ is not centrifugally folded with respect to $\mathfrak Q$.
 The gallery $(C_x^-=C'_0,C'_1,...,C'_r)$ is minimal if, and only if,  $c_j\not=1$ for any $j$ with $w_{\mathfrak Q}^{-1} \beta_j > 0$ or $\qb_j\not\in\QF_x$ and $u_{c_j(\alpha_{i_j})}\ne 1$ for any $j$ with $c_j=1$ and $w_{\mathfrak Q}^{-1} \beta_j < 0$.
\end{prop}

\begin{rema*} For $g_j$ as in equation (\ref{eq:Gal}) we may write $g_j=u_{c_j(\alpha_{i_j})} c_j$ (with $u_{c_j(\alpha_{i_j})}=1$ if $w_{\mathfrak Q}^{-1} \beta_j > 0  \hbox{ or } \qb_j\not\in\QF_x $).
 Then in the product $g_1\cdots g_j$ we may gather the $c_k$ on the right and, as $c_1\cdots c_k(\qa_{i_k})=-\qb_k$, we may write $g_1\cdots g_j=u_{-\qb_1}\cdots u_{-\qb_j}.c_1\cdots c_j$.
 Hence $C'_j:=g_1\cdots g_jC^-_x=u_{-\qb_1}\cdots u_{-\qb_j}C_j$.
 When $u_{-\qb_k}\neq1$ we have $\qb_k\in\QF_x$ and $w_{\mathfrak Q}^{-1} \beta_k<0$; so it is clear that $\qr_{\g Q}(C'_j)=C_j$.
 
 \par The gallery $(C_x^-=C'_0,C'_1,...,C'_r)$ (of type $\mathbf i$) is minimal if, and only if, we may also write (uniquely) $C'_j=u_{-\qa_{i_1}}.u_{r_{i_1}(-\qa_{i_2})}\cdots u_{r_{i_1}\cdots r_{i_{j-1}}(-\qa_{i_j})}.r_{i_1}\cdots r_{i_{j}}(C_x^-)=h_1\cdots h_j.r_{i_1}\cdots r_{i_{j}}(C_x^-)$ with $h_k=u_{r_{i_1}\cdots r_{i_{k-1}}(-\qa_{i_k})}\in\overline U_{r_{i_1}\cdots r_{i_{k-1}}(-\qa_{i_k})}$ (which fixes $C^-_x$). 
 In particular, $C'_j\in h_1\cdots h_j\A_x$. But this formula gives no way to know when $\qr_{\g Q}(C'_j)=C_j$.
  We know only that, when $\qb_k\not\in\QF_x$ \ie $r_{i_1}\cdots r_{i_{k-1}}(-\qa_{i_k})\not\in\QF_x$, we have necessarily $h_k=1$.
\end{rema*}

\begin{proof} As the type $\mathbf i$ of $(C_x^-=C'_0,C'_1,...,C'_r)$ is the type of a minimal decomposition, this gallery is minimal if, and only if, two consecutive chambers are different.
 So the last assertion is a consequence of the first ones.
 We prove these properties for $(C_x^-=C'_0,C'_1,...,C'_j)$ by induction on $j$.
 We write in the following just $H_j$ for
the common limit hyperplane $ H_{\beta_j}$ of $C_{j-1}$ and $C_j$ of type $i_j$.

\par There are five possible relative positions of
$\mathfrak Q$, $C_x^-$ and $C_1$ with respect to $H_1$ and we seek $C'_1$ with $\qr_{\g Q}(C'_1)=C_1$ and $\overline{C'_1}\supset\overline{C_x^-}\cap H_1$.

\par 0) $\qb_1=-c_1\qa_{i_1}\not\in\QF_x$, then $H_1$ is not a wall, each $C'_1$ with $\overline{C'_1}\supset\overline{C_x^-}\cap H_1$ is equal to $C^-_x$ or $r_{i_1}C^-_x$ and $C'_1$ or $C^-_x$ are contained in the same apartments.
 So $C'_1=C_1=c_1C^-_x$; $C_1$ and $\g Q$ are in $g_1\A_x=\A_x$ with $g_1=c_1$.
  When $C'_1=C^-_x$, we have $c_1=1$ and $\mathbf c$ is not centrifugally folded.

  \par We suppose now $\qb_1\in\QF_x$, so $H_1$ is a wall.

\par 1)  $C_x^-$ is on the same side of $H_1$ as $\mathfrak Q$ and $C_1$ not, then $c_1=r_{i_1}$, $\qb_1=\qa_{i_1}$, $w_{\mathfrak Q}^{-1} \beta_1< 0$,
$C'_1 = g_1C_x^- = u_{-\alpha_{i_1}}r_{i_1} C_x^- = u_{-\alpha_{i_1}} C_1$.
But $u_{-\alpha_{i_1}}$ pointwise stabilizes the halfspace bounded by $H_1$
containing $C_x^-$, hence $u_{-\alpha_{i_1}}(\mathfrak Q) = \mathfrak Q$ and $C'_1$ are in the apartment $g_1\A_x$.

\par 2) $\mathfrak Q$ and $C_x^- = C_1$ are separated by $H_1$, then $c_1=1$, $\qb_1=-\qa_{i_1}$, $w_{\mathfrak Q}^{-1} \beta_1< 0$,
$C'_1 = g_1C_x^- = u_{\alpha_{i_1}} C_x^-$ but $u_{\alpha_{i_1}}$ pointwise
stabilizes the halfspace bounded by $H_1$ not containing $C_x^-$, hence
$\mathfrak Q$ and $C'_1$ are in the apartment $g_1\A_x$.

\par 3)  $C_1$ is on the same side of $H_1$ as $\mathfrak Q$ and $C_x^-$ not, then $c_1=r_{i_1}$, $\qb_1=\qa_{i_1}$, $w_{\mathfrak Q}^{-1} \beta_1> 0$ and $C'_1$ has to be $C_1$ so $g_1=c_1=r_{i_1}$,
% $w_{\mathfrak Q}$ has a reduced decomposition that starts with $s_{i_1}$,
%$w_{\mathfrak Q} = s_{i_1}u$, so
$w_{\mathfrak Q}^{-1}(\alpha_{i_1}) > 0$,
moreover $\mathfrak Q$ and $C'_1 = r_{i_1}C_x^-=C_1$ are in the apartment $g_1\A_x$.

\par 4) $\mathfrak Q$ and $C_x^- = C_1$ are on the same side of $H_1$.
 Then $c_1=1$ and $w_{\mathfrak Q}^{-1} \beta_1> 0$; the gallery $\mathbf c$ is not centrifugally folded.
 So $\qr_{\g Q}(C'_1)=C_1$ implies $C'_1=C^-_x=g_1C^-_x$ with $g_1=c_1=1$ as in (\ref{eq:Gal}).
  But the gallery $(C_x^-=C'_0,C'_1,...,C'_j)$ cannot be minimal.

\par By induction we assume now that the chambers $\mathfrak Q$ and $C'_{j-1}=g_1\cdots g_{j-1} C_x^-$ are in the
apartment $A_{j-1} = g_1\cdots g_{j-1} \A_x$. Again, we have five possible relative positions
for $\mathfrak Q, C_{j-1}$ and $C_{j}$ with respect to $H_j$.
We seek $C'_j$ with $\qr_{\g Q}(C'_j)=C_j$ and $\overline{C'_j}\supset\overline{C'_{j-1}}\cap g_1\cdots g_{j-1}H_{\qa_{i_j}}$.

\par 0) $\qb_j=-c_1\cdots c_j\qa_{i_j}\not\in\QF_x$, then $H_j$ is not a wall, each $C'_j$ with $\overline{C'_j}\supset\overline{C'_{j-1}}\cap g_1\cdots g_{j-1}H_{\qa_{i_j}}$ is equal to $C'_{j-1}=g_1\cdots g_{j-1} C_x^-$ or $g_1\cdots g_{j-1}r_{i_j} C_x^-$; moreover $C'_j$ or $C'_{j-1}$ are contained in the same apartments.
 So $C'_j=g_1\cdots g_{j-1}c_jC^-_x$ and $\g Q$ are in $g_1\cdots g_{j}\A_x=g_1\cdots g_{j-1}\A_x$ with $g_j=c_j$.
  When $C'_j=C'_{j-1}$, we have $c_j=1$ and $\mathbf c$ is not centrifugally folded.

  \par We suppose now $\qb_j\in\QF_x$, so $H_j$ is a wall.

\par 1) $C_{j-1}$ is on the same side of $H_j=c_1\cdots c_{j-1}H_{\qa_{i_j}}$ as $\mathfrak Q$ and $C_{j}$ not, then $c_j=r_{i_j}$, $\qb_j=c_1\cdots c_{j-1}\qa_{i_j}$, $w_{\mathfrak Q}^{-1} \beta_j< 0$.
 Moreover $\mathfrak Q$ and $C'_{j-1}$ are on the same side of $g_1\cdots g_{j-1} H_{\qa_{i_j}}$ in $A_{j-1}$, and
$$
\begin{array}{rcl}
C'_j & =  &g_1\cdots g_{j-1} u_{-\alpha_{i_j}}r_{i_j} C_x^- \\
 & = & g_1\cdots g_{j-1} u_{-\alpha_{i_j}}r_{i_j} (g_1\cdots g_{j-1})^{-1} C'_{j-1}\\
 & = &  g_1\cdots g_{j-1} u_{-\alpha_{i_j}} (g_1\cdots g_{j-1})^{-1} g_1\cdots g_{j-1} r_{i_j} (g_1\cdots g_{j-1})^{-1} C'_{j-1},
\end{array}
$$  where $g_1\cdots g_{j-1} r_{i_j} (g_1\cdots g_{j-1})^{-1} C'_{j-1}$ is the chamber adjacent to
$C'_j$ along $g_1\cdots g_{j-1} H_{\qa_{i_j}}$ in $A_{j-1}$. Moreover, $g_1\cdots g_{j-1} u_{-\alpha_{i_j}} (g_1\cdots g_{j-1})^{-1}$ pointwise stabilizes the halfspace bounded by $ g_1\cdots g_{j-1} H_{\qa_{i_j}}$
containing $C'_{j-1}$ and $\mathfrak Q$. So $\mathfrak Q$ and $C'_j$ are in the apartment $g_1\cdots g_j \A_x$.

\par 2) $C_{j-1} = C_j$ and $\mathfrak Q$ are separated by $H_j$, then $c_j=1$, $\qb_j=-c_1\cdots c_{j-1}\qa_{i_j}$, $w_{\mathfrak Q}^{-1} \beta_j< 0$.
 Moreover $C'_{j-1}$ and $\mathfrak Q$ are separated by $g_1\cdots g_{j-1} H_{\qa_{i_j}}$ in $A_{j-1}$, and $\mathfrak Q$ and the
chamber
$$
g_1\cdots g_{j-1} r_{i_j} (g_1\cdots g_{j-1})^{-1} C'_{j-1}
$$
are on the same side of this wall. For $u_{\alpha_{i_j}}\ne 1$
$$
C'_j =  g_1\cdots g_{j-1} u_{\alpha_{i_j}} C_x^- = g_1\cdots g_{j-1}
u_{\alpha_{i_j}}(g_1\cdots g_{j-1})^{-1} C'_{j-1}
$$
is a chamber adjacent (or equal) to $C'_{j-1}$ along $g_1\cdots g_{j-1} H_{\qa_{i_j}}= g_1\cdots g_{j-1} u_{\alpha_{i_j}} H_{\qa_{i_j}}$
in $g_1\cdots g_j \A_x$ (with $g_j=u_{\alpha_{i_j}}$).

The root-subgroup  $g_1\cdots g_{j-1} U_{\alpha_{i_j}}(g_1\cdots g_{j-1})^{-1}$ pointwise
stabilizes the halfspace bounded by $ g_1\cdots g_{j-1} H_{\qa_{i_j}}$ and containing the chamber
$ g_1\cdots g_{j-1} r_{i_j} (g_1\cdots g_{j-1})^{-1} C'_{j-1}$. So $\mathfrak Q$ and $C'_j$
are in the apartment $g_1\cdots g_j \A_x$.

\par 3)  $C_j$ is on the same side of $H_j=c_1\cdots c_{j-1}H_{\qa_{i_j}}$ as $\mathfrak Q$ and $C_{j-1}$ not, then $c_j=r_{i_j}$, $\qb_j=c_1\cdots c_{j-1}\qa_{i_j}$, $w_{\mathfrak Q}^{-1} \beta_j> 0$.
 and so $C'_j =g_1\cdots g_{j-1} r_{i_j} C_x^-$. Whence
$\mathfrak Q$ and $C'_j$ are in the apartment $g_1\cdots g_j \A_x$.

\par 4) $C_{j-1} = C_j$ and $\mathfrak Q$ are on the same side of $H_j=c_1\cdots c_{j-1}H_{\qa_{i_j}}$, then $c_j=1$, $\qb_j=-c_1\cdots c_{j-1}\qa_{i_j}$ and $w_{\mathfrak Q}^{-1} \beta_j>0$.
 The gallery $\mathbf c$ is not centrifugally folded.
  So $\qr_{\g Q}(C'_j)=C_j$ implies $C'_j=C_{j-1}'=g_1\cdots g_jC^-_x$ with $g_j=c_j=1$ as in (\ref{eq:Gal}).
  But the gallery $(C_x^-=C'_0,C'_1,...,C'_j)$ cannot be minimal.
\end{proof}

\begin{coro}
\label{4.5}  If $\mathbf c \in \Gamma_{\mathfrak Q}^+(\mathbf i)$, then the number of elements in $\mathcal C^m_{\mathfrak Q}(\mathbf c)$ is:
$$
\sharp \mathcal C^m_{\mathfrak Q}(\mathbf c) = \prod_{k=1}^{t(\mathbf c)} q_{j_k} \times \prod_{l=1}^{r(\mathbf c)} (q_{j_l} - 1)
$$
 where $q_j=q_{x,\qb_j}=q_{x,\qa_{i_j}}\in\shq$, $t(\mathbf c) = \sharp\{j\mid c_j = r_{i_j}, \qb_j\in\QF_x \hbox{ and } w_{\mathfrak Q}^{-1} \beta_j < 0\}$ and
$r(\mathbf c) = \sharp\{j\mid c_j = 1, \qb_j\in\QF_x \hbox{ and } w_{\mathfrak Q}^{-1} \beta_j < 0\}$.
\end{coro}

%%%%%%%%%%%%%%%%%%%%%%%%%%%%%%%%%%%%%%%%%
\subsection{Galleries and opposite segment germs}
\label{4.6}

Suppose now $x\in\A\cap\SHI^+$. %In $\A_z^-$, let $C_{z}^-$ be the chamber corresponding to the sector $z+ (-C^v_f)$.
Let $\xi$ and $\eta$ be two segment germs in $\A_x^+$. Let $-\eta$ and $-\xi$ opposite respectively $\eta$ and $\xi$ in $\A_x^-$.
 Let $\mathbf i$ be the type of a minimal gallery between $C_x^-$ and $C_{-\xi}$, where $C_{-\xi}$ is the negative (local) chamber containing $-\xi$ such that $w(C_x^-, C_{-\xi})$ is of minimal length.
  Let $\mathfrak Q$ be a chamber of $\A_x^+$ containing $\eta$. We suppose $\qx$ and $\eta$ conjugated by $W^v_x$.
%We define $d^r_*(C_{z}^-, \xi)$, the codistance between $C_{z}^-$ and $\xi$, to be the codistance between $C_z^-$ and $C_\xi$, where $C_\xi$ is a positive chamber such that $\xi\in\overline{C_\xi}$ and $d^r_+(-C_z^-,C_\xi)$ is of minimal length (actually as everything is in the same apartment, $d^r_*(C_z^-, \xi) = d^r_*(C_z^-, C_\xi) = d^r_+(-C_z^-,C_\xi)$).

%A $(W^v_z, C_z^-)-$chain from $\xi$ to $\eta$ is a sequence
%$$
%\xi, \tau_1 \xi, \tau_2\tau_1 \xi, ..., \tau_n\cdots\tau_2\tau_1\xi = \eta,
%$$ where each $\tau_i$ is a reflection in $W^v_z$ that moves $\tau_{i-1}\cdots\tau_2\tau_1\xi$ away from $C_z^-$, which means that $d^r_*(C_z^-, \tau_i\tau_{i-1}\cdots\tau_2\tau_1\xi)\leq d^r_*(C_z^-, \tau_{i-1}\cdots\tau_2\tau_1\xi)$ for the Bruhat order in $W^v_x$. This is the same condition as the one in terms of positions with respect to the walls $M_i$ of $\tau_i$ we gave in \cite{BCGR11}, Section 4.2.

\begin{lemm*} The following conditions are equivalent:
%\label{le:OppGal}
\par (i) There exists an opposite $\zeta$ to $\eta$ in $\SHI_x^-$ such that $\rho_{\A_x, C^-_x} ( \zeta) = -\xi$.

\par (ii)  There exists a gallery $\mathbf c \in \Gamma_{\mathfrak Q}^+(\mathbf i)$ ending in $-\eta$.

\par (iii) $\qx\leq{}_{W^v_x}\,\eta$ (in the sense of \ref{de:Hecke}, with $\QF^+$ defined as in \ref{4.1} using $C^-_x$).

\par  Moreover the possible $\qz$ are in one-to-one correspondence with the disjoint union of the sets  $\mathcal C^m_{\mathfrak Q}(\mathbf c)$ for $\mathbf c$ in the set $\Gamma_{\mathfrak Q}^+(\mathbf i,-\eta)$ of galleries in $\Gamma_{\mathfrak Q}^+(\mathbf i)$ ending in $-\eta$.
  More precisely, if $\mathbf m\in \mathcal C_{\mathfrak Q}(\mathbf c)$ is associated to $(h_1,\cdots,h_r)$ as in remark \ref{4.4}, then $\qz=h_1\cdots h_r(-\qx)$.
\end{lemm*}

\begin{proof}
If $\zeta\in\SHI_x^-$ opposites $\eta$ and if $\rho_{\A_x, C^-_x} ( \zeta) = -\xi$, then any minimal gallery $\mathbf m = (C_x^-,M_1,...,M_r\ni\zeta)$ retracts onto a minimal gallery between $C_x^-$ and $C_{-\xi}$.
 So we can as well assume that $\mathbf m$ has type $\mathbf i = (i_1,...,i_r)$ and then $\qz$ determines $\mathbf m$.
 Now, if we retract $\mathbf m$ from $\mathfrak Q$, we get a gallery $\mathbf c = \rho_{\A_x, \mathfrak Q} (\mathbf m)$ in $\A_x^-$ starting at $C_x^-$, ending in $-\eta$ and centrifugally folded with respect to $\mathfrak Q$.

\par
Reciprocally, let $\mathbf c = (C_x^-,C_1,...,C_r) \in \Gamma_{\mathfrak Q}^+(\mathbf i)$, such that $-\eta\in C_r$. According to proposition and remark \ref{4.4},
there exists a minimal gallery $\mathbf m = (C^-_x,C'_1,...,C'_r)$ in the set $\mathcal C_{\mathfrak Q}(\mathbf c)$, and
the chambers $C_j'$ can be described by $C'_j = g_1\cdots g_j C_x^-=h_1\cdots h_j.r_{i_1}\cdots r_{i_j} C_x^-$ where each $h_k$ fixes $C^-_x$, hence $\rho_{\A_x, C^-_x}$ restricts on $C'_j$ to the action of $(h_1\cdots h_j)^{-1}$.

%where $g_j = c_j$ or $u_{c_j(\alpha_{i_j})} c_j $, with
%$c_j\not=1$ if $w_{\mathfrak Q}^{-1} \beta_j > 0$ or $\qb_j\not\in\QF_x$,
%and $u_{c_j(\alpha_{i_j})}\ne 1$ if $c_j=1$ and $w_{\mathfrak Q}^{-1} \beta_j < 0$.

\par Let $\zeta\in C'_r$ opposite $\eta$ in any apartment containing those two.
The minimality of the gallery $\mathbf m = (C^-_x,C'_1,...,C'_r)$ ensures that $\rho_{\A_x,C^-_x}(\qz)\in C_{-\qx}$; hence $\rho_{\A_x,C^-_x}(\zeta) = -\xi$ as they are both opposite $\eta$ up to conjugation by $W^v_x$.

\par So we proved the equivalence $(i)\iff(ii)$ and the last two assertions.

\par Now the equivalence $(i)\iff(iii)$ is proved in \cite[Prop. 6.1 and Th. 6.3]{GR08}: in this reference we speak of Hecke paths with respect to $-C^v_f$, but the essential part is a local discussion in $\SHI_x$ (using only $C^-_x$ and the twin building structure of $\SHI_x^\pm{}$) which gives this equivalence.
%(\footnote{peut-on faire une d\'emonstration directe de $(ii)\iff(iii)$ ? peut-on faire allusion \`a \cite{BCGR11} ?})
\end{proof}

\subsection{Liftings of Hecke paths}
\label{4.7}

Let $\pi$ be a $\lambda-$path from $z'\in Y^+$ to $y\in Y^{+}$ entirely contained in the Tits cone $\sht$, hence in a finite union of closed sectors $w\overline{C^v_f}$ with $w\in W^v$.
 By \cite[5.2.1]{GR08}, for each $w\in W^v$ there is only a finite number of $s\in]0,1]$ such that the reverse path $\bar\pi(t) = \pi(1-t)$ leaves, in $\qp(s)$, a wall positively with respect to $-w\overline{C^v_f}$, \ie this wall separates $\pi_-(s)$ from $-w\overline{C^v_f}$.
  Therefore, we are able to define $\ell\in\N$ and $0<t_1<t_2<\cdots<t_\ell\leq{}1$ such that the  $z_k = \pi(t_k)$, $k\in\{1,..., \ell\}$ are the only points in the path where at least one wall containing  $z_k $ separates $\qp_-(t_k )$ and the local chamber $\g c_-$ of \ref{1.9}.2.

 \par  For each $k\in\{1,..., \ell\}$ we choose for $C^-_{z_k}$ (as in \ref{4.1}) the germ in $z_k$ of the sector of vertex $z_k$ containing $\g c_-$.
  Let $\mathbf i_k$ be a fixed reduced decomposition of the element $w_-(t_k)\in W^v$ and let $\mathfrak Q_k$ be a fixed chamber in $\SHI^+_{z_k}$ containing $\eta_k = \pi_+(t_k)$.
  We note $-\qx_k=\qp_-(t_k)$. When $\qp$ is a Hecke path (or a billiard path as in \cite{GR08}), $\qx_k$ and $\eta_k$ are conjugated by $W^v_{z_k}$.

  \par When $\qp$ is a Hecke path with respect to $\g c_-$, $\{z_1,\cdots,z_\ell\}$ includes all points where the piecewise linear path $\qp$ is folded and, in the other points, all galleries in $\QG^+_{\g Q_k}(\mathbf i_k,-\eta_k)$ are unfolded.

Let $S_{\g c_-}(\pi, y)$ be the set of all segments $[z,y]$ such that $\rho_{\g c_-} ([z,y]) = \pi$.

\begin{theo}\label{4.8} $S_{\g c_-}(\pi, y)$ is non empty if, and only if,  $\pi$ is a Hecke path with respect to $\g c_-$. Then, we have a bijection
$$
S_{\g c_-}(\pi, y)\simeq \prod_{k=1}^\ell \coprod_{\mathbf c\in\Gamma_{\mathfrak Q_k}^+(\mathbf i_k,-\eta_k)} \mathcal C^m_{\mathfrak Q_k} (\mathbf c)
$$

\par In particular the number of elements in this set is a polynomial  in the numbers $q\in\shq$ with coefficients in $\Z_{\geq{}0}$ depending only on $\A$.
\end{theo}
\begin{NB} So the image by $\qr_{\g c_-}$ of a segment in $\SHI^+$ is a Hecke path with respect to $\g c_-$.
\end{NB}
\begin{proof} The restriction of $\qr_{\g c_-}$ to $\SHI_{z_k}$ is clearly equal to $\qr_{\A_{z_k},C^-_{z_k}}$; so the lemma \ref{4.6} tells that $\qp$ is a Hecke path with respect to $\g c_-$ if, and only if, each $\Gamma_{\mathfrak Q_k}^+(\mathbf i_k,-\eta_k)$ is non empty.

\par We set $t_0 = 0$ and $t_{\ell+1} = 1$. We shall build a bijection
from $S_{\g c_-}(\pi_{\vert[t_{n-1},1]}, y)$ onto $ \prod_{k=n}^\ell \coprod_{\mathbf c\in\Gamma_{\mathfrak Q_k}^+(\mathbf i_k,-\eta_k)} \mathcal C^m_{\mathfrak Q_k} (\mathbf c)$
 by decreasing induction on $n\in\{1,\cdots,\ell+1\}$.
  For $n=\ell+1$ and if $t_\ell\not=1$, no wall cutting $\qp([t_\ell,1])$ separates $y=\qp(1)$ from $\g c_-$; so a segment $s$ in $\SHI$ with $s(1)=y$ and $\rho_{\g c_-}\circ s=\qp$ has to coincide with $\qp$ on $[t_\ell,1]$.

  \par Suppose now that $s\in S_{\g c_-}(\pi_{\vert[t_{n},1]}, y)$ is determined, in the following way,  by a unique element in $ \prod_{k=n+1}^\ell \coprod_{\mathbf c\in\Gamma_{\mathfrak Q_k}^+(\mathbf i_k,-\eta_k)} \mathcal C^m_{\mathfrak Q_k} (\mathbf c)$:
  For an element $(\mathbf m_{n+1}, \mathbf m_{n+2},...,\mathbf m_\ell)$ in this last set, each $\mathbf m_k = (C^-_{z_k},C^k_1,...,C^k_{r_k})$ is a minimal gallery given by a sequence of elements $(h^k_1,...,h^k_{r_k})\in (\overline G_{z_k})^{r_k}$, as in remark \ref{4.4} and, for $t\in[t_n,t_{n+1}]$, we have $s(t)=(h^\ell_1...h^\ell_{r_\ell})\cdots(h^{n+1}_1...h^{n+1}_{r_{n+1}})\pi(t)$ where actually each $h^k_j$ is a chosen element of $U_{-r_{i_1}\cdots r_{i_{j-1}}(\qa_{i_{j}})}$ whose class in $\overline U_{-r_{i_1}\cdots r_{i_{j-1}}(\qa_{i_{j}})}$ is the $h^k_j$ defined above; in particular each  $h^k_j$ fixes $\g c_-$.

\par We note $g=(h^\ell_1...h^\ell_{r_\ell})\cdots(h^{n+1}_1...h^{n+1}_{r_{n+1}})\in G_{\g c_-}$. Then $g^{-1}s(t_n)=\qp(t_n)=z_n$.

\par If $s\in S_{\g c_-}(\pi_{\vert[t_{n-1},1]}, y)$ and $s_{\vert[t_{n},1]}$ is as above, then $g^{-1}s_-(t_n)$ is a segment germ in $\SHI^-_{z_n}$ opposite $g^{-1}s_+(t_n)=\qp_+(t_n)=\eta_n$ and retracting to $\qp_-(t_n)$ by $\qr_{\g c_-}$.
 By lemma \ref{4.6} and the above remark, this segment germ determines uniquely a minimal gallery $\mathbf m_n\in \mathcal C^m_{\mathfrak Q_n} (\mathbf c)$ with $\mathbf c\in\Gamma_{\mathfrak Q_n}^+(\mathbf i_n,-\eta_n)$.

\par Conversely such a minimal gallery $\mathbf m_n$ determines a segment germ $\qz\in\SHI^-_{z_n}$, opposite $\qp_+(t_n)=\eta_n$ such that $\qr_{\A_{z_n},C^-_{z_n}}(\qz)=\qp_-(t_n)$.
 By lemma \ref{4.6}, $\qz=(h^{n}_1...h^{n}_{r_{n}})\pi_-(t_n)$ for some well defined $(h^{n}_1,\cdots,h^{n}_{r_{n}})\in(\overline G_{z_n})^{r_n}$.
  As above we replace each $g^n_j$ by a chosen element of $G_{(z_n\cup\g c_-)}$ whose class in $\overline G_{z_n}$ is this $g^n_j$.
  As no wall cutting $[z_{n-1},z_n]$ separates $z_n=\qp(t_n)$ from $\g c_-$, any segment retracting by $\rho_{\g c_-}$ onto $[z_{n-1},z_n]$ and with $[z_n,x)=\qp_-(t_n)$ (resp. $=\qz$, $=g\qz$) is equal to $[z_{n-1},z_n]$ (resp. $(h^{n}_1...h^{n}_{r_{n}})[z_{n-1},z_n]$, $g(h^{n}_1...h^{n}_{r_{n}})[z_{n-1},z_n]$).
  We set $s(t)=(h^\ell_1...h^\ell_{r_\ell})\cdots(h^{n+1}_1...h^{n+1}_{r_{n+1}})(h^{n}_1...h^{n}_{r_{n}})\pi(t)$ for $t\in [t_{n-1},t_n]$.

\par With this inductive definition, $s$ is a $\ql-$path, $s(1)=y$, $\qr_{\g c_-}\circ s=\qp$ and $s_{\vert [t_{k-1},t_k]}$ is a segment $\forall k\in\{1,..., \ell+1\}$.
 Moreover, for $k\in\{1,..., \ell\}$, the segment germs $[s(t_k),s(t_{k+1}))$ and $[s(t_k),s(t_{k-1}))$ are opposite.
  By the following lemma this proves that $s$ itself is a segment.
%On the other side, we can successively apply the Lemma \ref{le:OppGal} to associate to a segment in $S(\pi,\nu)$ a sequence of minimal galleries, element of $\prod_{j=1}^k \coprod_{\mathbf c\in\Gamma_{\mathfrak Q}^+(\mathbf i_j)} \mathcal C^m_{\mathfrak Q_j} (\mathbf c)$. \marginpar{\`a d\' etailler !}
\end{proof}

\begin{lemm}\label{4.9} Let $x,y,z$ be three points in an ordered hovel $\SHI$, with $x\leq{}y\leq{}z$ and suppose the segment germs $[y,z)$ , $[y,x)$ opposite in the twin buildings $\SHI_y$.
 Then $[x,y]\cup[y,z]$ is the segment $[x,z]$.
\end{lemm}

\begin{proof} For any $u\in[y,z]$, we have $x\leq{}y\leq{}u\leq{}z$, hence $x$ and $[u,y)$ or $[u,z)$ are in a same apartment \cite[5.1]{R11}.
 As $[y,z]$ is compact we deduce that there are points $u_0=y,u_1,\cdots,u_\ell=z$ such that $x$ and $[u_{i-1},u_i]$ are in a same apartment $A_i$, for $1\leq{}i\leq{}\ell$.
 Now $A_1$ contains $x$ and $[y,u_1]$, hence also $[x,y]$ (axiom (MAO) of \ref{1.3}).
 But $[y,x)$ and $[y,u_1)=[y,z)$ are opposite, so $[x,y]\cup[y,u_1]=[x,u_1]$. The lemma follows by induction.
\end{proof}

\begin{rema}\label{4.10} The same things as above may be done for the retraction $\qr_{-\infty}$ instead of $\qr_{\g c_-}$: for all $x$ we choose $C^-_x=germ_x(x-C^v_f)$.
 For a $\ql-$path $\qp$ in $\A$ from $z'$ to $y$,  \cite[5.2.1]{GR08} tells that we have a finite number of points $z_k=\qp(t_k)$ where at least a wall is left positively by the path $\bar\qp(t)=\qp(1-t)$.
  We define as above $\mathbf i_k$, $\g Q_k$, $\eta_k$ and $\qx_k$. Now $S_{-\infty}(\pi, y)$ is the set of all segments $[z,y]$ such that $\qr_{-\infty}([z,y])=\qp$.

\par In \cite[Theorems 6.2 and 6.3]{GR08}, we have proven that $S_{-\infty}(\pi, y)$  is nonempty if, and only if, $\pi$ is a Hecke path with respect to $-C^v_f$.
Moreover, we have shown that, for $\SHI$ associated to a split Kac-Moody group over $\C(\!(t)\!)$, $S_{-\infty}(\pi, y)$ is isomorphic to a quasi-affine toric complex variety.
 The arguments above prove that, with our choice for $\SHI$, $S_{-\infty}(\pi, y)$ is finite, with the following precision (which generalizes to the Kac-Moody case some formulae of \cite{GL11}):
\end{rema}

\begin{prop}\label{4.11} Let $\qp$ be a Hecke path with respect to $-C^v_f$ from $z'$ to $y$. Then we have a bijection:
$$
S_{-\infty}(\pi, y)\simeq \prod_{k=1}^\ell \coprod_{\mathbf c\in\Gamma_{\mathfrak Q_k}^+(\mathbf i_k,-\eta_k)} \mathcal C^m_{\mathfrak Q_k} (\mathbf c)
$$

\par In particular the number of elements in this set is a polynomial  in the numbers $q\in\shq$ with coefficients in $\Z_{\geq{}0}$ depending only on $\A$.
\end{prop}

%  \par Il faudrait d\'emontrer les r\'esultats suivants qui sont similaires \`a [GR08; 5.9 et 6.3] ou les g\'en\'eralisent.

%  \begin{prop} Let $y_0,y_1\in Y^+$, $\ql\in Y^{++}$ and $\g c_-$ the negative fundamental alcove in $\A$. Then  (\footnote{Cela devient un corollaire de la proposition \ref{2.4}. })
 % \end{prop}

%    \begin{prop} Let $y_0,y_1\in Y^+$, $\ql\in Y^{++}$ and $\g c_-$  as in the proposition above and $\pi_1$ an Hecke path of shape $\ql$ with respect to $\g c_-$ starting in $y_0$ and ending in $y_1$.
%    Then the number of points $y\in\SHI^+_0$ with $\qr_-([y,y_1])=\pi_1$ is a product of elements $q_j$ or $q_j-1$ for $q_j\in\shq$ indexed by ...\qquad (\`a reformuler plus pr\'ecis\'ement!)
%  \end{prop}

\begin{theo}\label{4.12} Let $\ql,\qm,\qn\in Y^{++}$, $\g c_-$ the negative fundamental alcove and suppose $(\qa^\vee_i)_{i\in I}$ $\R^+-$free. Then

\par a) The number of Hecke paths of shape $\qm$ with respect to $\g c_-$ starting in $z'=w\ql$ (for some $w\in W^v$ fixing $0$) and ending in $y=\qn$ is finite.

b)  The structure constant $m_{\ql,\qm}(\qn)$ \ie the number of triangles $[0,z,\qn]$ in $\SHI$ with $d_v(0,z)=\ql$ and $d_v(z,\qn)=\qm$ is equal to:

\begin{equation}
\label{eq:SC}
m_{\ql,\qm}(\qn)=\sum_{w\in W^v/(W^v)_\lambda}\sum_\qp\prod_{k=1}^{\ell_\qp}\quad\sum _{\mathbf c\in\Gamma_{\mathfrak Q_k}^+(\mathbf i_k,-\eta_k)} \sharp \mathcal C^m_{\mathfrak Q_k} (\mathbf c)
\end{equation}
\par\noindent where $\qp$ runs over the set of Hecke paths of shape $\qm$ with respect to $\g c_-$ from $w\ql$ to $\qn$ and $\ell_\qp$, $\Gamma_{\mathfrak Q_k}^+(\mathbf i_k,-\eta_k)$ and $\mathcal C^m_{\mathfrak Q_k} (\mathbf c)$ are defined as above for each such  $\qp$.

c) In particular the structure constants of  the Hecke algebra $\shh_R$ are polynomials  in the numbers $q\in\shq$ with coefficients in $\Z_{\geq{}0}$ depending only on $\A$.
\end{theo}

\begin{proof} We saw in \ref{2.2a}.1 that $m_{\ql,\qm}(\qn)$ is the number of $z\in\SHI_0^+$ such that  $d_v(0,z)=\ql$ and $d_v(z,\qn)=\qm$.
 Such a $z$ determines uniquely a Hecke path $\qp=\rho_{\g c_-} ([z,\qn])$ of shape $\qm$ with respect to $\g c_-$ from $z'=\rho_{\g c_-} (z)$ to $\qn$.
  But $d_v(0,z)=\ql$ and $0\in\g c_-$, so $d_v(0,z')=\ql$ \ie $z'=w\ql$ with $w\in W^v$. So the formula (\ref{eq:SC}) follows from theorem \ref{4.8}.

  \par We know already that $m_{\ql,\qm}(\qn)$ is finite (\ref{2.4}) and $S_{\g c_-}(\pi, y)\not=\emptyset$  (theorem \ref{4.8}), hence a) is clear.
   Now c) follows from corollary \ref{4.5}
\end{proof}

%\begin{rema}\label{4.13} The commutativity of $\widehat\shh_R$ or $\shh_R$ corresponds to polynomials identities (depending on $\A$) in the variables $q\in\shq$. In the homogeneous case (where all $q\in\shq$ are equal) and for $\A$ associated to a RGS as in \ref{3.1} we saw that these identities are verified for $q$ a power of a prime number; hence for any $q$.

%\par So, for any choice of a homogeneous hovel with the same $\A$ and $Y$, $\widehat\shh_R$ or $\shh_R$ are still commutative.
%\end{rema}

\section{Satake isomorphism}\label{s5}

In this section, we prove the Satake isomorphism. From now on, we assume that the $\alpha_i^\vee$'s are free.

\par We denote by $U^-$ the fixator in $G$ of the sector germ $\g S_{-\infty}$, \ie any $u\in U^-$ has to fix pointwise a sector $x-C^v_f\subset\A$.
 By definition, for $z\in\SHI$, $\qr_{-\infty}(z)$ is the only point of the orbit $U^-.z$ in $\A$.

\subsection{The module of functions on the type $0$ vertices in $\A$}\label{5.1}

Let $\A_0 = \nu(N)\cdot 0 = Y\cdot 0$ be the set of vertices of type $0$ in $\A$. Note that $\A_0$ can be identified to the set of horocycles of $U^-$ in $\SHI_0$, i.e. to $\SHI_0/U^-$, via the retraction $\rho_{-\infty}$.
We consider first $\widehat{\mathscr F}=\widehat{\mathscr F}_R = \mathscr F(\A_0,R)$, the set of functions on $\A_0$ with values in a ring $R$. Equivalently, $\widehat{\mathscr F}$ can be identified with the set of $U^--$invariant functions on $\SHI_0$. 

\par For $\qm\in Y$, we define $\chi_\qm\in\widehat{\mathscr F}$ as the characteristic function of $U^-.\qm$ in $\SHI_0$ (or $\{\qm\}$ in $Y$).
 Then, any $\chi\in\widehat{\mathscr F}_R$ may be written $\chi=\sum_{\qm\in Y}a_\qm\chi_\qm$ with $a_\qm\in R$. We set $supp(\chi)=\{\qm\mid a_\qm\not=0\}$.
Now, let 
$$
\mathscr F =\mathscr F_R = \{\chi\in\widehat{\mathscr F}\mid supp(\chi)\subset \cup_{j=1}^n(\mu_j-Q^\vee_+) \text{ for some }\qm_j\in\A_0\}
$$ be the set of functions on $\SHI_0$ with almost finite support.

\par We define also the following completion of the group algebra $R[Y]$:
$$R[[Y]]=\{f=\sum_{y\in Y} a_ye^y\mid supp(f)=\{y\in Y\mid a_y\not=0\}\subset \cup_{j=1}^n(\mu_j-Q^\vee_+) \text{ for some }\qm_j\in\A_0\}$$
it is clearly a commutative algebra (with $e^y.e^z=e^{y+z}$). Actually, it is the Looijenga's coweight algebra, see Section 4.1 in \cite{Loo}.

\par The formula $(f.\chi)(\qm)=\sum_{y\in Y}a_y\chi(\qm-y)$, for $f=\sum a_ye^y\in R[[Y]]$, $\chi\in\SHF$ and $\qm\in Y$, defines an element $f.\chi\in\SHF$; in particular $e^y.\chi_\qm=\chi_{\qm+y}$.
 Clearly, the map $R[[Y]]\times\SHF\to\SHF$, $(f,\chi)\mapsto f.\chi$ makes $\SHF$ into a free $R[[Y]]-$module of rank $1$, with any $\chi_\qm$ as basis element.

\begin{defiprop}\label{5.2}
The map 
$$
\begin{array}{rcl}
\mathscr F\times\shh & \to & \mathscr F\\
(\chi,\varphi) & \mapsto & \chi*\varphi,
\end{array}
$$ where, for $x\in\SHI_0$, $(\chi*\varphi) (x) = \sum_{y\in\SHI_0} \chi(y)\varphi^\SHI(y,x)$, defines a right action of $\shh$ on $\mathscr F$ that commutes with the actions of $Z = \{n\in N\mid \nu(n)\in Y\}$ and (more generally) $R[[Y]]$.
\end{defiprop}

\begin{proof}
It is relatively clear that $\chi*\varphi$ is a function on $\SHI_0/U^-$ and that the map indeed defines an action. Let us check that this action commutes with the one of $Z$. Let $t\in Z$ and $x\in\SHI_0$, then 
$$
\begin{array}{rcl}
(\chi*\varphi)(tx) & = & \sum_{y\in\SHI_0} \chi(y)\varphi^\SHI(y,tx)\\ 
 & = & \sum_{y'\in\SHI_0} \chi(ty')\varphi^\SHI(ty',tx) \qquad (y = ty')\\
 & = & \sum_{y'\in\SHI_0} \chi(ty')\varphi^\SHI(y',x)\\
 & = & ((\chi\circ t)*\varphi)(x).
\end{array}
$$ 
So, $(\chi\circ t)*\varphi= (\chi*\varphi)\circ t$.
 For $\qn(t)=\qm\in Y$ and $\chi\in\SHF$, we have clearly $\chi\circ t=e^{-\qm}.\chi$.
 As a formal consequence, the right action of $\shh$ commutes with the left action of $R[[Y]]$.

\par The difficult point is to show that the support condition is satisfied. 
%To do so, let $\mu\in Y$ and let $\chi_\mu$ be the characteristic function of $U^-\cdot \mu$. Then, 
For any $\lambda\in Y^{++}$, and any $\nu\in Y$, 
$$
\begin{array}{rcl}
(\chi_\mu*c_\lambda)(\nu) & = & \sum_{y\in\SHI_0} \chi_\mu(y)c_\lambda^\SHI(y,\nu)\\ 
 & = & \sharp \{y\in\SHI_0\mid \rho_{-\infty}(y) = \mu \ \text{ and }\ d^v(y,\nu) = \lambda\}
\end{array}
$$ 
The latest is also the cardinality of the set of all segments $[y,\nu]$ in $\SHI$ ($y\leq \nu$) of ``length'' $\lambda$ such that $y\in U^-\cdot \mu$. In addition, since the action of $\shh$ commutes with the one of $Z$, we set $n_\lambda(\nu-\mu) = (\chi_\mu*c_\lambda)(\nu) $. Then $n_\lambda(\nu-\mu) = \sum_{\pi} \sharp S_{-\infty}(\pi, \nu)$ where the sum runs over the set of Hecke $\lambda-$paths with respect to $-C_f^v$ from $\mu$ to $\nu$ (see \ref{4.10} for the definition of $ S_{-\infty}(\pi, \nu)$).

Now, Lemma \ref{2.3} b) shows that $n_\lambda(\nu-\mu)\ne 0$ implies $\nu-\mu\leq_{Q_+} \lambda$. Moreover, if $\nu = \lambda + \mu$, then $n_\lambda(\lambda) = 1$. Therefore, we get 
\begin{equation}\label{eq:FH}
\chi_\mu*c_\lambda = \sum_{\nu\leq_{Q^\vee} \lambda +\mu} n_\lambda(\nu-\mu) \chi_\nu= \chi_{\ql+\qm}+\sum_{\nu<_{Q^\vee} \lambda +\mu} n_\lambda(\nu-\mu) \chi_\nu.
\end{equation}
This formula shows that, for any $\varphi\in\shh$ with $supp(\varphi)\subset \cup_{i=1}^n(\lambda_i-Q^\vee_+)$ and any $\xi\in\mathscr F$ with $supp(\chi)\subset \cup_{j=1}^n(\mu_j-Q^\vee_+)$, the support of $\chi*\varphi$ is contained in $\cup_{i,j}(\lambda_i + \mu_j-Q^\vee_+)$. More precisely, for any $\nu\in \cup_{i,j}(\lambda_i + \mu_j-Q^\vee_+)$ there exists a finite number of $\lambda\in supp(\varphi)$ and $\mu\in supp(\chi)$ such that $\nu\leq_{Q_+} \lambda +\mu$. Hence, $\chi*\varphi$ is well defined.
\end{proof}

\subsection{The Satake isomorphism}\label{5.3}

\subsubsection{The  morphism $\shs_*$}

As $\SHF$ is a free $R[[Y]]-$module of rank one, we have $End_{R[[Y]]}(\SHF)=R[[Y]]$. So the right action of $\shh$ on the $R[[Y]]-$module $\SHF$ gives an algebra homomorphism $\shs_*:\shh\to R[[Y]]$ such that $\chi*\qf=\shs_*(\qf).\chi$ for any $\qf\in\shh$ and any $\chi\in\SHF$.

\par As $e^\nu.\chi_\mu=\chi_{\qm+\qn}$, equation (\ref{eq:FH}) gives 
$$\shs_*(c_\ql)= \sum_{\nu\leq_{Q^\vee} \lambda } n_\lambda(\nu) e^\nu= e^{\ql}+\sum_{\nu<_{Q^\vee} \lambda} n_\lambda(\nu) e^\nu$$

We shall modify $\shs_*$ by some character to get the Satake isomorphism.

\subsubsection{The module $\qd$}

We define a map $\qd:Q^\vee\to\R_+^*\,,\,\sum_{i\in I}\,a_i\qa_i^\vee\mapsto\prod_{i\in I}\,(q_iq_i')^{a_i}$, where $q_i,q_i'\in\shq\subset\N$ are as in the beginning of Section \ref{s4}.
 We extend this homomorphism and its square root to $Y$ (as $\R_+^*$ is uniquely divisible).
 So, we get homomorphisms $\qd,\qd^{1/2}:Y\to\R_+^*$ and $\qd=\qd\circ\qn,\qd^{1/2}=\qd^{1/2}\circ\qn:Z\to \R_+^*$.
 
 \par We made a choice for $\qd$. But we shall see in theorem \ref{5.4} that the expected properties depend only on $\qd\rest{\,Q^\vee}$.

\par In the classical case, where $G$ is a split semi-simple group and $\SHI$ its Bruhat-Tits building, we have $q_i=q_i'=q$ for any $i\in I$. Hence, if we set $\mu =\sum_{i\in I}\,a_i\qa_i^\vee$, $\qd^{1/2}(\qm)=q^{\sum a_i}=q^{\qr(\qm)}$ where $\qr$ is the half sum of positive roots.

\subsubsection{The Satake isomorphism}

\par From now on, we suppose that the algebra $R$ contains the image of $\qd^{1/2}$ in $\R_+^*$. We define 
$$
\shs(c_\ql)=\sum_{\qm\leq_{Q^\vee}\ql}\,\qd^{1/2}(\qm)n_\ql(\qm)e^\qm=\qd^{1/2}(\ql)e^\ql+\sum_{\qm<_{Q^\vee}\ql}\,\qd^{1/2}(\qm)n_\ql(\qm)e^\qm
$$ and extend it to formal combinations of the $c_\ql$ with almost finite support.

\par We get thus an algebra homomorphism $\shs:\shh\to R[[Y]]$ called the {\it Satake isomorphism}, as it is one to one:

\par For $\qf=\sum_\ql a_\ql c_\ql\in\shh$, we have $\shs(\qf)=\sum_\ql\,a_\ql\big(\qd^{1/2}(\ql)e^\ql+\sum_{\qm<_{Q^\vee}\ql}\,\qd^{1/2}(\qm)n_\ql(\qm)e^\qm \big)$.
 If $\qf\not=0$ and $\ql_0$ is a maximum element in $supp(\qf)$, then $\ql_0$ is also a maximum element in $supp(\shs(\qf))$ and $\shs(\qf)\not=0$.

\begin{remas*} 

a) So we already know that $\shh$ is commutative.
 
b) In the classical case where $G$ is a split semi-simple group, $\shs(c_\ql)$ is defined as an integral over a maximal unipotent subgroup, we choose here $U^-$.
  The Haar measure $du$ on $U^-$ is chosen to give volume $1$ to $K\cap U^-$, and, for an element $t$ in the torus $Z$, the formula for changing variables is given by $d(tut^{-1})=\qd(t)^{-1}du$. 
   So the classical formula for the Satake isomorphism given \eg in \cite[(19) p 146]{Ca79} when $\qn(t)=\qm$, is:
 
 $$
\begin{array}{rcl}
\shs(c_\ql)(t) &= \qquad\qd(t)^{1/2}\int_{U^-}\,c_\ql^G(ut)du\quad\; & = \; \qd(t)^{1/2}\int_{U^-}\,c_\ql^\SHI(0,ut.0)du\\
\qquad &= \;\qd(t)^{1/2}\int_{U^-}\,c_\ql^\SHI(u^{-1}.0,t.0)du&=\;\qd(t)^{1/2}\sum_{y\in U^-.0}\,c_\ql^\SHI(y,\qm)\\
\qquad  &=\;\qd(t)^{1/2}\sum_{y\in\SHI_0}\,\chi_0(y).c_\ql^\SHI(y,\qm)&=\;\qd(t)^{1/2}(\chi_0*c_\ql)(\qm)
\end{array}
$$
 This is the same formula as ours.
 \end{remas*}
 
\subsubsection{$W^v-$invariance} 

There is an action of $W^v$ on $Y$, hence on $R[Y]$ by setting $w.e^\ql=e^{w\ql}$ for $w\in W^v$ and $\ql\in Y$.
  This action does not extend to $R[[Y]]$, but we define $R[[Y]]^{W^v}=\{f=\sum a_\ql e^\ql\in R[[Y]] \mid a_\ql=a_{w\ql}, \forall\ql\in Y, \forall w\in W^v\}$.
  This is a subalgebra of $R[[Y]]$ and actually the image of the Satake isomorphism (see Theorem \ref{5.4}).

\begin{rema*}

 Let $C^\vee=\{\qp\in V^*\mid\qa_i^\vee(\qp)\geq0,\forall i\in I\}$ and $\sht^\vee=\cup_{w\in W^v}\,wC^\vee$ be the fundamental dual chamber and the dual Tits cone in $V^*$.
  By definition, for $f\in R[[Y]]$ and $\qp\in C^\vee$, $\qp(supp(f))$ is bounded above.
  Hence, for $f\in R[[Y]]^{W^v}$, $\qp(supp(f))$ is also  bounded above for any $\qp\in\sht^\vee$.
  We know that the dual cone of $\overline\sht^\vee$ is the closed convex hull $\overline\Gamma$ of the set $\QD_+^{\vee im}\cup\{0\}$, where $\QD_+^{\vee im}\subset Q_+^\vee$ is the set of positive imaginary roots in the dual system of roots $\QD^\vee$, \cite[5.8]{K90}.
  So, the only directions along which points in $supp(f)$ (for $f\in R[[Y]]^{W^v}$) may go to infinity are the directions in $-\overline\Gamma$.
\end{rema*}

\begin{theo}\label{5.4} 
The Hecke algebra $\shh_R$ is isomorphic via $\shs$ to the commutative algebra $R[[Y]]^{W^v}$ of Weyl invariant elements in $R[[Y]]$.
\end{theo}
\begin{proof} 
As $\shs(c_\ql)=\sum_{\qm\leq_{Q^\vee}\ql}\,\qd^{1/2}(\qm)n_\ql(\qm)e^\qm$ we only have to prove that, for $w\in W^v$, $\qd^{1/2}(\qm)n_\ql(\qm)=\qd^{1/2}(w\qm)n_\ql(w\qm)$ or $n_\ql(w\qm)=n_\ql(\qm)\qd^{1/2}(\qm-w\qm)$.
 It is sufficient to prove this for $w=r_i$ a fundamental reflection, hence to prove that $n_\ql(r_i\qm)=n_\ql(\qm)\qd^{1/2}(\qm-r_i\qm)=n_\ql(\qm)\qd^{1/2}(\qa_i(\qm)\qa_i^\vee)$.
  By the given definition of $\qd$, the wanted formula is:
  
  \begin{equation}\label{eq:inv}
 n_\ql(r_i\qm)=n_\ql(\qm)\Big(\sqrt{q_iq_i'}\Big)^{\qa_i(\qm)}
\end{equation}
 The proof of this formula is postponed to the following subsections, starting with \ref{5.10}.
 One can already notice that $\qa_i(\qm)$ is an integer. If it is odd, since any $t\in Z$ with $\qn(t)=\qm$ exchanges the walls $M(\qa_i,0)$ and $M(\qa_i,\qa_i(\qm))$, hence $q_i=q_i'$.
  So, in any case $\big(\sqrt{q_iq_i'}\big)^{\vert\qa_i(\qm)\vert}$ is an integer.
  
  \par Once the formula (\ref{eq:inv}) is proved we know that $\shs(\shh)\subset R[[Y]]^{W^v}$.
   For $f=\sum a_\qm e^\qm\in R[[Y]]^{W^v}$ with $supp(f)\subset\cup_{j=1}^r\,(\ql_j-Q_+^\vee)$, we shall build a sequence $\qf_n$ in $\shh$ such that $supp(f-\shs(\qf_n))\subset\cup_{j=1}^r\,(\ql_j-Q_{+n}^\vee)$ and $supp(\qf_{n+1}-\qf_n)\subset Y^{++}\cap(\cup_{j=1}^r\,(\ql_j-Q_{+n}^\vee))$, where $Q_{+n}^\vee=\{\sum_{i\in I}n_i\qa_i^\vee\in Q_+^\vee\mid\sum n_i\geq n\}$.
   Then, the limit $\qf$ of this sequence exists in $\shh$ and $\shs(\qf)=f$.
   So, $\shs$ is onto.
   
   \par We build the sequence by induction. We set $\qf_0=0$.
   If $\qf_0,\cdots,\qf_n$ are given as above, we set $\{\qm_1,\cdots,\qm_s\}= supp(f-\shs(\qf_n))\setminus\cup_{j=1}^r\,(\ql_j-Q_{+(n+1)}^\vee)$. 
   For any $w\in W^v$, $w\qm_k\in supp(f-\shs(\qf_n))\subset\cup_{j=1}^r\,(\ql_j-Q_{+n}^\vee)$, so $w\qm_k$ cannot be strictly greater than $\qm_k$ for $\leq_{Q^\vee}$; this proves that $\qm_k\in Y^{++}$.
   So we define $\qf_{n+1}=\qf_n-\sum_{k=1}^s\,a_{\qm_k}(f-\shs(\qf_n))\qd(\qm_k)^{-1/2}c_{\qm_k}$.
   As $\shs(c_\ql)=\qd^{1/2}(\ql)e^\ql+\sum_{\qm<_{Q^\vee}\ql}\,\qd^{1/2}(\qm)n_\ql(\qm)e^\qm$, this $\qf_{n+1}$ is suitable.
\end{proof}

\begin{rema*}
Suppose $G$ is a split Kac-Moody group as in Section \ref{s3}. And consider the complex Kac-Moody algebra $\mathfrak g^\vee$ associated with $G^\vee$, the Langlands dual of $G$. Let $\mathfrak h^\vee = \mathbb C\otimes_{\mathbb Z} Y$ be the Cartan subalgebra of $\mathfrak g^\vee$. Let Rep$(\mathfrak g^\vee)$ be the category of $\mathfrak g^\vee-$modules $V$ such that $V$ is $\mathfrak h^\vee-$diagonalizable, the weight spaces $V_\lambda$ are finite dimensional and the set $\mathscr P(V)$ of weights of $V$ satisfies $\mathscr P(V)\subset \cup_{j=1}^r\,(\ql_j-Q_+^\vee)$, for some $\lambda_j$. One can check that Rep$(\mathfrak g^\vee)$ is stable by tensoring, hence, we can consider its Grothendieck ring $K(\mathfrak g^\vee)$. Now, the map $[V]\mapsto \sum_\lambda (\dim V_\lambda)e^\lambda$ is an isomorphism from $K(\mathfrak g^\vee)$ onto $\mathbb C[[Y]]^{W^v}$. Therefore, by composing it with $\shs$, we get an isomorphism between $\shh_{\mathbb C}$ and $K(\mathfrak g^\vee)$.
\end{rema*}

\subsection{Extended tree associated to $(\A,\qa_i)$} \label{5.10}
\par We consider the vectorial panel $-F^v(\{i\})$ in $-\overline{C^v_f}$ and its support the vectorial wall $Ker(\qa_i)$. Their respective directions are a panel $\g F_\infty$ in a wall $M_\infty$, in the twin buildings $\SHI^{\pm\infty}$ at infinite of $\SHI$ \cite[3.3, 3.4, 3.7]{R11}.

\par The germs of the sector panels in $\SHI$ of direction $\g F_\infty$ are the points of an (essential) affine building  $\SHI(\g F_\infty)$, which is of rank $1$ \ie a tree \cite[4.6]{R11}.

\par The union $\SHI(M_\infty)$ of the apartments in $\SHI$ containing a wall of direction $M_\infty$ is an inessential affine building whose essential quotient is $\SHI(\g F_\infty)$  \cite[4.9]{R11}.
More precisely $\SHI(M_\infty)$ may be identified with the product of the tree $\SHI(\g F_\infty)$ and an affine space quotient of $\A$.

\par The canonical apartment of $\SHI(M_\infty)$ is $\A$ endowed with a smaller set of walls: uniquely the walls of direction $Ker(\qa_i)$. 
 As we chose $\SHI$ semi-discrete (\ref{1.2}), this is a locally finite set of hyperplanes; hence $\SHI(M_\infty)$ is discrete and $\SHI(\g F_\infty)$ a discrete tree (not an $\R-$tree).
 By \cite[2.9]{R11} the valences of these walls are the same in $\SHI(M_\infty)$ and in $\SHI$, \ie $1+q_i$ and $1+q_i'$; hence $\SHI(\g F_\infty)$ is a semi-homogeneous tree of valences $1+q_i$ and $1+q_i'$.
 By definition, $0\in\A$ is in a wall of valence $1+q_i$.
 
 \par We asked that the stabilizer $N$ of $\A$ in $G$ is positive and type preserving (\ref{1.3}) \ie acts on $V=\vect{\A}$ via $W^v$.
  So, the stabilizer in $W^v$ of $M_\infty$ is $\{1,r_i\}$ and $M_\infty$ determines in $V$ a supplementary vectorial subspace of dimension one : $M_\infty^\perp=Ker(1+r_i)$.
  The affine space $\A$ decomposes as the product of the affine space $E=\A/M_\infty^\perp$ with associated vector space $Ker(\qa_i)$ and an affine line ($=\A/Ker(\qa_i)$).
  This decomposition is canonical \ie invariant by the stabilizer $N(M_\infty)$ of $M_\infty$ in $N$.
   As a consequence we get the decomposition  $\SHI(M_\infty)=E\times\SHI(\g F_\infty)$ which is canonical \ie invariant by the stabilizer $G(M_\infty)$ of $M_\infty$ in $G$.
    Moreover $G(M_\infty)$ acts on $E$ by translations only.

\begin{rema*} Suppose $\g G$ is an almost split Kac-Moody group over a local field $\shk$ and $\SHI$ its associated hovel as in \cite{R13}.
 Then the stabilizer $G(\g F_\infty)$ of $\g F_\infty$ in $G$ is a parabolic subgroup, endowed with a Levi decomposition $G(\g F_\infty)=G(M_\infty)\ltimes U(\g F_\infty)$ (with $U(\g F_\infty)\subset U^-$)  and $\SHI(M_\infty)$ (resp. $\SHI(\g F_\infty)$) is the extended (resp. essential) Bruhat-Tits building associated to the reductive group of rank one $G(M_\infty)$, embedded in $\SHI$ \cite[6.12.2]{R13}.
  Any orbit of $U(\g F_\infty)$ in $\SHI$ meets $\SHI(M_\infty)$ in one and only one point.
  
  \par The tree $\SHI(\g F_\infty)$ is a piece of the polyhedral ``compactification'' of $\SHI$ (a true compactification when $\g G$ is reductive).
  With the notation of \cite{R13}, $\SHI(M_\infty)$ (resp. $\SHI(\g F_\infty)$) is the fa\c cade $\SHI(\g G,\shk,\overline{\overline\A})_{\g F_\infty}$ (resp. $\SHI(\g G,\shk,\overline\A^e)_{\g F_\infty}$).
\end{rema*}

\subsection{Parabolic retraction} \label{5.11}

  Let $x$ be a point in $\SHI$. There is a unique sector-panel $x+\g F_\infty$ of vertex $x$ and direction $\g F_\infty$ \cite[4.7.1]{R11}.
  The germ of this sector-panel is a point in $\SHI(\g F_\infty)$, the {\it projection} $pr_{\g F_\infty}(x)$ of $x$ onto $\SHI(\g F_\infty)$, \cf \cite{Ch10}, \cite{Ch11} or \cite[4.3.5]{R13} in the Kac-Moody case.
  
  \par Let $A_x$ be an apartment in $\SHI$ containing $x$ and $\g F_\infty$, hence $x+\g F_\infty$ and $germ_\infty(x+\g F_\infty)$.
  But this germ is in an apartment $B_x$ of $\SHI(M_\infty)$ (axiom (MA3) applied to $germ_\infty(x+\g F_\infty)$ and a sector of direction $C^v_f$) and there exists an isomorphism $\psi_x$ of $A_x$ onto $B_x$ fixing this germ (axiom (MA2)).
  One writes $\qr(x)=\psi_x(x)\in \SHI(M_\infty)$. We have thus defined the {\it retraction} $\qr=\qr_{\g F_\infty,M_\infty}$ {\it of} $\SHI$ {\it onto} $\SHI(M_\infty)$ {\it with center} $\g F_\infty$. We shall now verify that $\qr(x)$ does not depend on the choices made.
  
  \par By definition, $\qr(x)$ is in the hyperplane $H_x$ of $B_x$ containing $germ_\infty(x+\g F_\infty)$ and of direction $M_\infty$, this $H_x$ does not depend on the choice of $B_x$.
  Moreover for two choices $\psi_x:A_x\to B_x$ and $\psi'_x:A_x'\to B_x$, $\psi'_x\circ\psi_x^{-1}$ is the identity on $germ_\infty(x+\g F_\infty)$ hence on $H_x$.
  It is now clear that $\psi_x(x)=\psi'_x(x)$.
  Actually $\qr(x)$ may also be defined in the following simple way: there exist $y,z\in(x+\g F_\infty)\cap B_x$ such that $y$ is the middle of $[x,z]$ in $A_x$, then $\qr(x)$ is the point of $H_x\subset B_x$ such that $y$ is the middle of $[\qr(x),z]$ in $B_x$.
  
  \begin{rema*} It is possible to prove that the image by $\qr$ of a preordered segment is a polygonal line and, in some generalized sense, a Hecke path.
    \end{rema*}
  
 \subsection{Factorization of $\qr_{-\infty}$} \label{5.12}
 
 The panel $\g F_\infty$ is in the closure of the chamber $\g C_{-\infty}$ of $\SHI^{-\infty}$ associated to $-C^v_f$.
  So this chamber or the associated sector-germ $\g S_{-\infty}$ determines an end of the tree $\SHI(\g F_\infty)$ \cite[4.6]{R11} \ie a sector-germ $\g S'$ in $\SHI(M_\infty)$: $\g S'$ is one of the two sector-germs in $\A$ (considered as an apartment of $\SHI(M_\infty)$ with its small set of walls), each element in $\g S'$ contains an half apartment of equation $\qa_i(y)\leq k$ with $k\in \Z$.
   We write $\qr'_{-\infty}$ the retraction of $\SHI(M_\infty)$ onto $\A$ with center  $\g S'$.
   
   \begin{lemm*} The retraction $\qr_{-\infty}$ factorizes through $\qr$ : $\qr_{-\infty}=\qr'_{-\infty}\circ\qr$.
   \end{lemm*}
   \begin{proof} For $x\in\SHI$, one chooses an apartment $A_x$ containing $x$ and $\g C_{-\infty}$, hence the sector $x+\g C_{-\infty}$, its sector-germ $\g S_{-\infty}$ and its panel $x+\g F_\infty$.
   One chooses also an apartment $B_x$ of $\SHI(M_\infty)$ containing $germ_\infty(x+\g F_\infty)$ and  $\g S_{-\infty}$. 
   Hence, $A_x$ and $B_x$ contain both $germ_\infty(x+\g F_\infty)$ and  $\g S_{-\infty}$; by axiom (MA4) there exists an isomorphism $\psi_x$ of $A_x$ onto $B_x$ fixing these two germs.
   By the definition of the parabolic retraction, in \ref{5.11}, $\qr(x)=\psi_x(x)$.
   
   \par Now the apartments $A_x$ and $B_x$ of $\SHI(M_\infty)$ contain both $\g S_{-\infty}$, hence $\g S'$.
   So there is an isomorphism $\qth: B_x\to\A$ fixing $\g S'$, hence $\g S_{-\infty}$.
   As $\qr(x)\in B_x$, one has $\qr'_{-\infty}\circ\qr(x)=\qth(\qr(x))=\qth\circ\psi_x(x)$ and this is $\qr_{-\infty}(x)$ as $\qth\circ\psi_x:A_x\to\A$ is an isomorphism fixing $\g S_{-\infty}$.
   \end{proof}
  
\subsection{Counting} \label{5.13}

We want to prove  equation (\ref{eq:inv}):
 $ n_\ql(r_i\qm)=n_\ql(\qm)\big(\sqrt{q_iq_i'}\big)^{\qa_i(\qm)}$ 
  for $\ql\in Y^{++}$ and $\qm\in Y$, where $n_\ql(\qm)$ is the number of points $y\in\SHI_0$ such that $\qr_{-\infty}(y)=-\qm$ and $d^v(y,0)=\ql$, \cf \ref{5.2}.
  For $z\in\SHI(M_\infty)$ one writes $p_\ql(z)\in\Z_{\geq0}\cup\{\infty\}$ for the number of points $y\in\SHI_0$ such that $\qr(y)=z$ and $d^v(y,0)=\ql$. 
  By lemma \ref{5.12}, $n_\ql(\qm)$ is the sum of $p_\ql(z)$ for $z\in\SHI(M_\infty)\cap\SHI_0$ such that $\qr'_{-\infty}(z)=-\qm$.
  
  \par Let $M_0=0+M_\infty=Ker(\qa_i)$ be the wall in $\A$ of direction $M_\infty$ containing $0$.
  Its fixator $G(M_0)$ ($\subset G(M_\infty)$) acts transitively on the apartments of $\SHI$ or $\SHI(M_\infty)$ containing it (by axiom (MA4), as $M_0$ is the enclosure of two sector panel germs).
  Moreover $G(M_\infty)$ fixes $\g F_\infty$, hence $\qr$ is $G(M_\infty)-$equivariant.
  As a consequence, the weight function $p_\ql$ is constant on the orbits of $G(M_0)$ in $\SHI(M_\infty)\cap\SHI_0$.
  Hence $n_\ql(\qm)=\sum_\QO\,p_\ql(\QO)n^\QO(-\qm)$, where the sum runs over the orbits $\QO$ of $G(M_0)$ in $\SHI(M_\infty)\cap\SHI_0$ and $n^\QO(\qn)$ is the number of points in the orbit $\QO$ such that $\qr'_{-\infty}(z)=\qn$.
  
  \par To prove formula (\ref{eq:inv}), it is sufficient to prove for any orbit $\QO$ as above and any $\qn\in Y$ that: 
  $$ n^\QO(r_i\qn)=n^\QO(\qn)\Big(\sqrt{q_iq_i'}\Big)^{-\qa_i(\qn)}$$
  
  \par We saw, in \ref{5.10}, that $G(M_\infty)$ leaves the decomposition $\SHI(M_\infty)=\SHI(\g F_\infty)\times E$ invariant and acts on $E$ by translations.
  But $G(M_0)$ fixes $M_0\ni0$, so it acts trivially on $E$.
  As $G(M_0)$ is transitive on the apartments containing $M_0$, an orbit $\QO$ is a set $S_r\times\{e\}$ where $S_r$ is the sphere of radius $r\in\Z_{\geq0}$ and center $0$ in the tree $\SHI(\g F_\infty)$.
  The apartment $\A$ (with its small set of walls) is the product $(\R,\Z)\times E$, where $\qa_i$ is the projection of $\A$ onto the one dimensional apartment $\R$ with vertex set $\Z$.
  
  \par So, the above formula, hence Formula (\ref{eq:inv}) and Theorem \ref{5.4} are consequences of the following proposition.
  The fact that $q_i=q_i'$ when $m=\qa_i(\qn)$ is odd, was explained in the proof of \ref{5.4}.

\subsection{The tree case} \label{5.14}

Let  $\T$ be a (discrete) semi-homogeneous tree. Let $\A\simeq\R$ be an apartment in $\T$ whose vertices are identified with $\Z$. The valence of the vertex $s\in\Z$ is $1+q$ (resp. $1+q'$) if $s$ is even (resp. odd).
 Let $-\infty$ be the end of $\A$ corresponding to integers converging towards $-\infty$.
 Let $\qr'$ be the retraction of $\T$ onto $\A$ with center $-\infty$.
 For $m\in\Z\subset\A$ and $r\in\Z_{\geq0}$ we write $n_r(m)$ the number of vertices in the sphere $S_r$ of center $0$ and radius $r$ in $\T$ such that $\qr'(z)=m$.
 
 \par If $m$ is odd we ask that $q=q'$.
 
 \begin{prop*} One has $n_r(m)=n_r(-m)(\sqrt{qq'})^m$.
 \end{prop*}
  \begin{rema*} This formula is equivalent to the $W^v(\T)-$invariance of the image of the Satake isomorphism for the Bruhat-Tits tree $\T$. 
  As this invariance is known, the following proof is not necessary; we give it for the convenience of the reader.
  
  \par For a Bruhat-Tits tree $\SHI=\T$, there are two choices for $\SHI_0$ (and $Y$): the set of vertices at even distance from $0$ or the full set of vertices.
  In this last case, we have to allow $m$ to be odd and we see below that the hypothesis $q=q'$ is necessary to get the formula.
  So, even for classical Bruhat-Tits buildings, to get the good image for the Satake isomorphism, $\SHI_0$ cannot be any $G-$stable set of special vertices (we chose $\SHI_0$ to be a $G-$orbit).
 \end{rema*}
  \begin{proof} For $z\in S_r$, let $s_z\in\Z$ be the vertex of $\A$ such that $[0,s_z]=[0,z]\cap\A$.
  Then $\qr'(z)=s_z+(r-\vert s_z\vert)\in\Z$.
  
  \par We can calculate the number $n_r(m)$ of vertices  $z\in S_r$  such that $\qr'(z)=m$:
  
  \par{\it First case}: $s_z\geq0\iff\qr'(z)=r$. So $n_r(r)=qq'qq'\cdots$ ($r$ factors).
  
    \par{\it Second case}: $-r\leq s_z<0\iff\qr'(z)<r$ and then $\qr'(z)=r+2s_z$ \ie $s_z=(\qr'(z)-r)/2$. The number $n_r(m)$ is then:
    $$
\begin{array}{rcl}
1\qquad &  & \text{ if} \quad m=s_z=-r \\ 
(q-1)q'qq'\cdots &(r+s_z=(r+m)/2 \text{ factors) }  & \text{ if} \quad s_z\in]-r,0[  {\text\; is\;even}\\ 
(q'-1)qq'q\cdots &(r+s_z=(r+m)/2 \text{ factors) }  & \text{ if} \quad s_z\in]-r,0[   {\text\; is\;odd}
\end{array}
$$ 
\par It is now easy to compare $n_r(m)$ and $n_r(-m)$. We get the wanted formula, using that $q=q'$ when $m$ is odd.
 \end{proof}

%\newpage
%%%%%%%%%%%%%%%%%%%%%%%%%%%%%%%%%%%%%%%%%%%%%%%
%%%%%%%%%%%%%%%%%%%%%%%%%%%%%%%%%%%%%%%%%%%%%%%%%%   \ar[u]@{^{(}->}

\medskip

Institut \'Elie Cartan Nancy, UMR 7502, Universit\'e de Lorraine, CNRS

Boulevard des aiguillettes, BP 70239, 54506 Vand\oe uvre l\`es Nancy Cedex (France)

Stephane.Gaussent@iecn.u-nancy.fr \qquad Guy.Rousseau@iecn.u-nancy.fr

\bigskip
\noindent {\bf Acknowledgement:} The first author acknowledges support of the ANR grants ANR-09-JCJC-0102-01 and ANR-2010-BLAN-110-02.

\end{document}